\documentclass[11pt,final]{article}
\usepackage{a4}
\usepackage{amsmath}%
\usepackage{amsthm}%
\usepackage{amstext}%
\usepackage{amssymb}%
\usepackage{showkeys}%
\usepackage{epsfig}%
\usepackage{cite}
\usepackage{tikz}
\usepackage{subfig}
\usepackage{caption}
\usepackage{multirow}
\usepackage{booktabs}
\usepackage{floatrow}
\usepackage{comment}
\usepackage{listings}

\newtheorem{theorem}{Theorem}

\newtheorem{axiom}{Axiom}

\newtheorem{definition}[axiom]{Definition}
\newtheorem{exm}{Example}
\newenvironment{example}{\exm\rm}{\endexm}

\newtheorem{lemma}[theorem]{Lemma}
\newtheorem*{nt}{Notations}
\newenvironment{notation}{\nt\rm}{\endnt}

\newtheorem{proposition}[theorem]{Proposition}
\newtheorem{rem}{Remark}
\newenvironment{remark}{\rem\rm}{\endrem}

\newcommand{\R}{\mathbb{R}}%
\newcommand{\e}{\varepsilon}%
\newcommand{\ol}{\overline}%

\newcommand{\n}{{\nabla}}
\newcommand{\p}{{\partial}}

\newcommand{\To}{\longrightarrow}
\def\a{\alpha}

\def\e{\epsilon}

\def\g{\gamma}
\def\G{\Gamma}

\def\l{\lambda}
\def\<{\langle}
\def\>{\rangle}
\DeclareMathOperator*\sqri{sqri}%
\DeclareMathOperator*\dom{dom}%
\DeclareMathOperator*\oR{\overline{\R}}%
\DeclareMathOperator*\gr{Gr}%
\DeclareMathOperator*\id{Id}%
\DeclareMathOperator*\prox{prox}%
\DeclareMathOperator*\argmin{argmin}

\DeclareMathOperator*\proj{proj}
\title{A primal-dual dynamical approach to structured convex minimization problems}

\author{Radu Ioan Bo\c{t} \thanks{University of Vienna, Faculty of Mathematics, Oskar-Morgenstern-Platz 1, A-1090 Vienna, Austria,
email: radu.bot@univie.ac.at. Research partially supported by FWF (Austrian Science Fund), project I 2419-N32.
} \and
Ern\"{o} Robert Csetnek \thanks {University of Vienna, Faculty of Mathematics, Oskar-Morgenstern-Platz 1, A-1090 Vienna, Austria,
email: ernoe.robert.csetnek@univie.ac.at. Research supported by FWF (Austrian Science Fund), project P 29809-N32.
}
 \and Szil\'{a}rd Csaba L\'{a}szl\'{o} \thanks{Technical University of Cluj-Napoca, Department of Mathematics, Memorandumului 28, Cluj-Napoca, Romania, email: szilard.laszlo@math.utcluj.ro.
 This work was supported by a grant of Ministry of Research and Innovation, CNCS - UEFISCDI, project number PN-III-P1-1.1-TE-2016-0266, within PNCDI III.}}

\begin{document}
\maketitle
\noindent \textbf{Abstract.} In this paper we propose a primal-dual dynamical approach to the minimization of a structured convex function consisting of a smooth term, a nonsmooth term, and the composition of another nonsmooth term with a linear continuous operator. 
In this scope we introduce a dynamical system for which we prove that its trajectories asymptotically converge to a saddle point of the Lagrangian of the underlying convex minimization problem as time tends to infinity. 
In addition, we provide rates for both the violation of the feasibility condition by the ergodic trajectories and the convergence of the objective function along  these ergodic trajectories to its minimal value. Explicit time discretization of the dynamical system results in a numerical algorithm which is a combination of the linearized proximal method of multipliers and the proximal ADMM algorithm.
\vspace{1ex}

 \vspace{1ex}


\section{Introduction and preliminaries}\label{sec1}

For ${\mathcal H}$ and ${\mathcal G}$ real Hilbert spaces, we consider the convex minimization problem
\begin{equation}\label{primal}
\inf_{x\in{\mathcal H}}f(x)+h(x)+g(Ax),
\end{equation}
where $f:{\mathcal H}\To\oR =\R \cup \{\pm \infty\}$ and $g:{\mathcal G}\To\oR$ are proper, convex and lower semicontinuous functions, $h:\mathcal{H}\To\R$ is a convex and Fr\'echet differentiable function with $L_h$-Lipschitz continuous gradient $(L_h \geq 0)$, 
i.e. $\|\n h(x)-\n h(y)\|\le L_h\|x-y\|$ for every $x,y\in \mathcal{H}$,  and $A:{\mathcal H}\To{\mathcal G}$ is a continuous linear operator.

Problem \eqref{primal} can be rewritten as
\begin{equation}\label{primaleq}
\inf_{\stackrel{(x,z)\in{\mathcal H}\times{\mathcal G}}{Ax-z=0}}f(x)+h(x)+g(z).
\end{equation}
Obviously, $x^*\in{\mathcal H}$ is an optimal solution of \eqref{primal} if and only if $(x^*,z^*)\in{\mathcal H}\times{\mathcal G}$ is an optimal solution of \eqref{primaleq} and $Ax^*=z^*$.

Based on this reformulation of problem \eqref{primal} we define its Lagrangian
$$l:{\mathcal H}\times{\mathcal G}\times{\mathcal G}\To\oR,\,l(x,z,y)=f(x)+h(x)+g(z)+\<y,Ax-z\>.$$
An element $(x^*,z^*,y^*)\in {\mathcal H}\times{\mathcal G}\times{\mathcal G}$ is said to be a saddle point of the Lagrangian $l$, if
$$l(x^*,z^*,y)\le l(x^*,z^*,y^*)\le l(x,z,y^*),\,\forall (x,z,y)\in {\mathcal H}\times{\mathcal G}\times{\mathcal G}.$$
It is known that $(x^*,z^*,y^*)\in {\mathcal H}\times{\mathcal G}\times{\mathcal G}$ is a saddle point of  $l$ if and only if  $x^*$ is an optimal solution of \eqref{primal}, $Ax^*=z^*$, and $y^*$ is an optimal solution of the Fenchel dual to problem \eqref{primal}, which reads
\begin{equation}\label{dual}
\sup_{y\in{\mathcal G}} \left (-(f^*\square h^*)(-A^*y)-g^*(y)\right).
\end{equation}
In this situation the optimal objective values of \eqref{primal} and \eqref{dual} coincide.

In the formulation of \eqref{dual}, 
$$f^* : {\cal H} \rightarrow \oR, \ f^*(u)=\sup_{x\in{\mathcal H}}(\<u,x\>-f(x)), \ \ h^* : {\cal H} \rightarrow \oR, \ h^*(u)=\sup_{x\in{\mathcal H}}(\<u,x\>-h(x)),$$
and
$$g^* : {\cal G} \rightarrow \oR, \ g^*(y)=\sup_{z\in{\mathcal G}}(\<y,z\>-g(z)),$$
denote the conjugate functions of $f, h$ and $g$, respectively, and $A^* : {\cal G} \rightarrow {\cal H}$ denotes the adjoint operator of $A$. The infimal convolution $f^*\square h^*:\mathcal{H}\to\oR$ of the functions $f^*$ and $h^*$ is defined by
$$(f^*\square h^*)(x)=\inf_{y\in\mathcal{H}}(f^*(y)+h^*(x-y)).$$

It is also known that $(x^*,z^*,y^*)\in {\mathcal H}\times{\mathcal G}\times{\mathcal G}$ is a saddle point of the Lagrangian $l$ if and only if it is a solution of the following system of primal-dual optimality conditions
$$\left \{ \begin{array}{l}
0 \in \partial f(x) + \nabla h(x) + A^*y\\
Ax = z, Ax \in \partial g^*(y).
\end{array} \right. $$
We recall that the convex subdifferential of the function $f :{\cal H} \rightarrow \oR$ at $x \in {\cal H}$ is defined by
$\partial f(x) = \{u \in {\cal H}: f(x') - f(x) \geq \langle u, x'-x \rangle \ \forall x' \in {\cal H}\} $, for $f(x) \in \R$, and by $\partial f(x) = \emptyset$, otherwise.

A saddle point of the Lagrangian $l$ exists whenever the primal problem \eqref{primal} has an optimal solution and the so-called Attouch-Br\'{e}zis regularity
condition
$$0\in\sqri(\dom g-A(\dom f))$$ holds. Here,
$$\sqri Q:=\{x\in Q:\cup_{\l>0}\l(Q-x)\mbox{ is a closed linear subspace of }\mathcal{G}\}$$
denotes the strong quasi-relative interior of a set $Q\subseteq\mathcal{G}$. We refer the reader
to \cite{bauschke-book, b-hab, Zal-carte} for more insights into the world of regularity conditions and convex duality theory.

Let $S_+(\mathcal{H})$ denote the family of continuous linear operators $U:\mathcal{H}\To\mathcal{H}$ which are self-adjoint and positive semidefinite. For $U\in S_+(\mathcal{H})$ we introduce the following seminorm on ${\cal H}$:
 $$\|x\|^2_U=\<x,Ux\> \ \forall x\in\mathcal{H}.$$
This introduces on $S_+(\mathcal{H})$ the following partial ordering:
for $U_1,U_2\in S_+(\mathcal{H})$
$$U_1\succcurlyeq U_2\Leftrightarrow \|x\|^2_{U_1}\ge\|x\|^2_{U_2} \ \forall x\in\mathcal{H}.$$
For $\a>0$ fixed, let be
$$P_\a(\mathcal{H})=\{U\in S_+(\mathcal{H}): U\succcurlyeq \a I\},$$
where $I:\mathcal{H}\To\mathcal{H},\,I(x)=x,$ denotes the identity operator on ${\cal H}$.

The subject of our investigations in this paper will be the following dynamical system, for which we will show that it asymptotically approaches the set of solutions of the primal-dual pair of optimization problems \eqref{primal}-\eqref{dual} 
\begin{equation}\label{ADMMdysy-subdiff}
\left\{
\begin{array}{llll}
\dot{x}(t)+x(t)\in\left(\partial f +cA^*A+M_1(t)\right)^{-1}\left(M_1(t)x(t)+cA^*z(t)-A^*y(t)-\n h(x(t))\right)\\
\\
\dot{z}(t)+z(t)\in \left(\partial g +cI+M_2(t)\right)^{-1}\left(M_2(t)z(t)+cA(\gamma\dot x(t)+x(t))+y(t)\right)\\
\\
\dot{y}(t)=c A(x(t)+\dot{x}(t))-c (z(t)+\dot{z}(t))\\
\\
x(0)=x^0\in{\mathcal H},\,z(0)=z^0\in{\mathcal G},\,y(0)=y^0 \in{\mathcal G},
\end{array}
\right.
\end{equation}
where $c >0$, $\gamma \in [0,1]$, and $M_1:[0,+\infty)\To S_+(\mathcal{H})$ and $M_2:[0,+\infty)\To S_+(\mathcal{G})$ are such that
$$(Cstrong) \quad \text{there exists} \ \a>0 \ \text{such that} \ cA^*A+M_1(t)\in P_{\a}(\mathcal{H}) \quad \forall t\in[0,+\infty).$$

One of the motivation for the study of this dynamical system comes from the fact that, as we will see in Remark \ref{rem-discretization}, it provides through explicit time discretization a numerical algorithm which is a combination of the linearized proximal method of multipliers and the proximal ADMM algorithm.

In the next section we will show the existence and uniqueness of strong global solutions for the dynamical system \eqref{ADMMdysy-subdiff} in the framework of the Cauchy-Lipschitz Theorem. In Section \ref{sectech} we will prove some technical results, which will play an important role in the asymptotic analyis.
In Section \ref{sec3} we will investigate the asymptotic behaviour of the trajectories as the time tends to infinity. By carrying out a Lyapunov analysis and by relying on the continuous variant of the Opial Lemma, we are able to prove that the trajectories generated by \eqref{ADMMdysy-subdiff}  
asymptotically convergence to a saddle point of the Lagrangian $l$. Furthermore, we provide convergence rates of ${\cal O}(\frac{1}{t})$ for the violation of the feasibility condition by ergodic trajectories and the convergence of the objective function along  these ergodic trajectories to its minimal value. 

The approach of optimization problems by dynamical systems has a long tradition. Crandall and Pazy considered in  \cite{crandall-pazy1969} dynamical systems governed by subdifferential operators (and more general by maximally monotone operators) in Hilbert spaces, 
addressed questions like the existence and uniqueness of solution trajectories, and 
related the latter to the theory of semi-groups of nonlinear contractions. Br\'{e}zis \cite{brezis} studied the asymptotic behaviour of the trajectories for dynamical systems governed by convex subdifferentials, and Bruck carried out in \cite{bruck} a similar analysis for maximally monotone operators. Dynamical systems defined via resolvent/proximal evaluations of the governing operators have enjoyed much attention in the last years, as they result by explicit time discretization in relaxed versions of standard numerical algorithms, with high flexibility and good numerical performances. Abbas and Attouch introduced in \cite{abbas-att-arx14} a forward-backward dynamical system, by extending to more general optimization problems an approach proposed by Antipin in \cite{antipin} and Bolte in \cite{bolte-2003} on a gradient-projected dynamical system associated to the minimization of a smooth convex function over a convex closed set. Implicit dynamical systems were considered also in \cite{b-c-dyn-KM} in the context of monotone inclusion problems. A  dynamical system of forward-backward-forward type was considered in \cite{bb-ts-cont}, while, a dynamical system of Douglas-Rachford type was recently introduced in  \cite{cmt}. 

It is important to notice that the approaches mentioned above have been introduced in connection with the study of ``simple" monotone inclusion and convex minimization problems. They rely on straightforward splitting strategies and cannot be efficiently used when addressing structured minimization problems, like \eqref{primal}, which need to be addressed from a primal and a dual perspective, thus, require for tools and techniques from the convex duality theory. The dynamical approach we introduce and investigate in this paper is, to our knowledge, the first meant to address structured convex minimization problems in the spirit of the full splitting paradigm.

\begin{remark}\label{rem-discretization}
The first inclusion in \eqref{ADMMdysy-subdiff} can be equivalently written as
\begin{equation}\label{1eq}0\in\p f(\dot{x}(t)+x(t))+c A^*A(\dot{x}(t)+x(t))+M_1(t)\dot{x}(t)-(c A^*z(t)-A^*y(t)-\n h(x(t))) \ \forall t \in [0,+\infty),\end{equation}
while the second one as
\begin{equation}\label{2eq}
0\in \p g(\dot{z}(t)+z(t))+c(\dot{z}(t)+z(t))-cA(\g\dot{x}(t)+x(t))-y(t)+M_2(t)\dot{z}(t) \ \forall t \in [0,+\infty).
\end{equation}

The explicit discretization of \eqref{1eq} with respect to the time variable $t$ and constant step equal to $1$ yields the iterative scheme
$$0\in \frac{1}{c}\p f(x^{k+1})+A^*Ax^{k+1}+\frac{M_1^k}{c}(x^{k+1}-x^k)-A^*z^k+\frac{A^*}{c}y^k+\frac{1}{c}\n h(x^k) \ \forall k \geq 0.$$
By convex subdifferential calculus, one can easily see that this can be for every $k \geq 0$ equivalently written as
$$0\in\p\left(f(x)+\<x-x^k,\n h(x^k)\>+\frac{c}{2}\left\|Ax-z^k+\frac{y^k}{c}\right\|^2+\frac12\|x-x^k\|^2_{M_1^k}\right)\bigg|_{x=x^{k+1}}$$
and, further, as
\begin{equation*}
x^{k+1}\in\argmin_{x\in\mathcal{H}}\left(f(x)+\<x-x^k,\n h(x^k)\>+\frac{c}{2}\left\|Ax-z^k+\frac{y^k}{c}\right\|^2+\frac12\|x-x^k\|^2_{M_1^k}\right).
\end{equation*}
Similarly, \eqref{2eq} leads for every $k \geq 0$ to
$$0 \in \p\left(g(z)+\frac{c}{2}\left\|A(\g x^{k+1}+(1-\g)x^k)-z+\frac{y^k}{c}\right\|^2+\frac12\|z-z^k\|^2_{M_2^k}\right)\bigg|_{z=z^{k+1}}, $$
which is nothing else than
\begin{equation*}
z^{k+1}\in\argmin_{z\in \mathcal{G}}\left(g(z)+\frac{c}{2}\left\|A(\g x^{k+1}+(1-\g)x^k)-z+\frac{y^k}{c}\right\|^2+\frac12\|z-z^k\|^2_{M_2^k}\right).
\end{equation*}
Here, $(M_1^k)_{k \geq 0}$ and $(M_2^k)_{k \geq 0}$ are two operator sequences in $S_+(\mathcal{H})$ and $S_+(\mathcal{G})$, respectively.

Thus the dynamical system \eqref{ADMMdysy-subdiff} leads through explicit time discretization to a numerical algorithm, which, for a starting point $(x^0, z^0, y^0) \in {\cal H} \times {\cal G} \times {\cal G}$, generates a sequence $(x^k, z^k, y^k)_{k \geq 0}$ for every $k \geq 0$ as follows
\begin{equation}\label{ADMMdiscrete}
\left\{
\begin{array}{llll}
x^{k+1}\in\argmin\limits_{x\in\mathcal{H}}\left(f(x)+\<x-x^k,\n h(x^k)\>+\frac{c}{2}\left\|Ax-z^k+\frac{y^k}{c}\right\|^2+\frac12\|x-x^k\|^2_{M_1^k}\right)\\
\\
z^{k+1}\in \argmin\limits_{z\in \mathcal{G}}\left(g(z)+\frac{c}{2}\left\|A(\g x^{k+1}+(1-\g)x^k)-z+\frac{y^k}{c}\right\|^2+\frac12\|z-z^k\|^2_{M_2^k}\right)\\
\\
y^{k+1}=y^k+c(Ax^{k+1}-z^{k+1}).
\end{array}
\right.
\end{equation}
The algorithm \eqref{ADMMdiscrete} is a combination of the linearized proximal method of multipliers and the proximal ADMM algorithm. 

Indeed, in the case when $\gamma =1$, \eqref{ADMMdiscrete} becomes the proximal ADMM algorithm with variable metrics from \cite{Ban-Bot-Csetnek} (see, also, \cite{b-c-acm}). If, in addition, $h=0$ and the operator sequences $(M_1^k)_{k \geq 0}$ and $(M_2^k)_{k \geq 0}$ are constant, then \eqref{ADMMdiscrete} becomes the proximal ADMM algorithm investigated in \cite[Section 3.2]{s-teb} (see, also, \cite{fpst}). It is known that the proximal ADMM algorithm can be seen as a generalization of the full splitting primal-dual algorithms of Chambolle-Pock (see \cite{ch-pck}) and Condat-Vu (see \cite{condat2013, vu}).

On the other hand, in the case when $\gamma =0$, \eqref{ADMMdiscrete} becomes an extension of the  linearized proximal method of multipliers of Chen-Teboulle (see \cite{chen-teb}, \cite[Algorithm 1]{s-teb}).
\end{remark}

In the following remark we provide a particular choice for the maps $M_1$ and $M_2$, which transforms \eqref{ADMMdysy-subdiff} into a dynamical system of primal-dual type formulated in the spirit of the full splitting paradigm.

\begin{remark}\label{rem-cont-primal-dual}
For every $t \in [0,+\infty)$, define $$M_1(t)=\frac{1}{\tau(t)}I-cA^*A \ \mbox{and} \ M_2(t)=0,$$
where $\tau(t)>0$ is such that $c\tau(t)\|A\|^2\le1$. 

Let $t \in [0, +\infty)$ be fixed. In this particular setting, \eqref{1eq} is equivalent to
$$\left(\frac{1}{\tau(t)}I-c A^*A\right)x(t)+c A^*z(t)-A^*y(t)-\n h(x(t))\in\frac{1}{\tau(t)}\dot{x}(t)+\frac{1}{\tau(t)}{x}(t)+\p f(\dot{x}(t)+x(t))$$
and further to
$$\dot{x}(t)+x(t)=(I+\tau(t)\p f)^{-1}\left((I-c\tau(t)A^*A)x(t)+c\tau(t) A^*z(t)-\tau(t)A^*y(t)-\tau(t)\n h(x(t))\right).$$
In other words,
\begin{equation*}
\dot{x}(t)+x(t)=\prox\nolimits_{\tau(t)f}\big((I-c\tau(t)A^*A)x(t)+c\tau(t) A^*z(t)-\tau(t)A^*y(t)-\tau(t)\n h(x(t))\big),
\end{equation*}
where
\begin{equation*}\label{intr-prox-def}  \prox\nolimits_{\kappa} : {\mathcal H} \rightarrow {\mathcal H}, \quad \prox\nolimits_{\kappa}(x)=
\argmin_{y\in {\mathcal H}}\left \{\kappa(y)+\frac{1}{2}\|x-y\|^2\right\}=(I+\partial \kappa)^{-1}(x),
\end{equation*}
denotes the proximal point operator of a proper, convex  and lower semicontinuous function $\kappa: {\cal H}\rightarrow \oR$.

On the other hand, relation \eqref{2eq} is equivalent to
$$\dot{y}(t)+y(t)+c(\g-1)A\dot{x}(t)\in\p g(\dot{z}(t)+z(t)),$$
hence,
$$\dot{z}(t)+z(t)\in\p g^*(\dot{y}(t)+y(t)+c(\g-1)A\dot{x}(t)).$$
This is further equivalent to
$$A(\g\dot{x}(t)+x(t))+\frac{1}{c}y(t)\in \frac{1}{c}\dot{y}(t)+\frac{1}{c}y(t)+(\g-1)A\dot{x}(t)+\p g^*(\dot{y}(t)+y(t)+c(\g-1)A\dot{x}(t))$$
and further to
$$\dot{y}(t)+y(t)+c(\g-1)A\dot{x}(t)=(I+c\p g^*)^{-1}(cA(\g\dot{x}(t)+x(t))+y(t)).$$
In other words,
\begin{equation*}
\dot{y}(t)+y(t)+c(\g-1)A\dot{x}(t)=\prox\nolimits_{cg^*}\big(cA(\g\dot{x}(t)+x(t))+y(t)\big).
\end{equation*}
Consequently, in this particular setting, the dynamical system \eqref{ADMMdysy-subdiff} can be equivalently written as
\begin{equation}\label{ADMMdysyM1}
\left\{
\begin{array}{llll}
\dot{x}(t)+x(t)=\prox\nolimits_{\tau(t)f}\big((I-c\tau(t)A^*A)x(t)+c\tau(t) A^*z(t)-\tau(t)A^*y(t)-\tau(t)\n h(x(t))\big)\\
\\
\dot{y}(t)+y(t)+c(\g-1)A\dot{x}(t)=\prox\nolimits_{cg^*}\big(cA(\g\dot{x}(t)+x(t))+y(t)\big)\\
\\
\dot{y}(t)=c A(x(t)+\dot{x}(t))-c (z(t)+\dot{z}(t))\\
\\
x(0)=x^0\in{\mathcal H},\,z(0)=z^0\in{\mathcal G},\,y(0)=y^0 \in{\mathcal G}.
\end{array}
\right.
\end{equation}
Let us also mention that when $h=0$ and $\gamma =1$ the dynamical system \eqref{ADMMdysyM1} reads
\begin{equation}\label{ADMMdysyM1-equiv}
\left\{
\begin{array}{llll}
\dot{x}(t)+x(t)=\prox\nolimits_{\tau(t)f}\left(x(t)-\tau(t)A^*(y(t)+cAx(t)-cz(t))\right)\\
\\
\dot{y}(t)+y(t)=\prox\nolimits_{cg^*}(y(t)+cA(\dot{x}(t)+x(t)))\\
\\
\dot{y}(t)=c A(x(t)+\dot{x}(t))-c (z(t)+\dot{z}(t))\\
\\
x(0)=x^0\in{\mathcal H},\,z(0)=z^0\in{\mathcal G},\,y(0)=y^0\in{\mathcal G}.
\end{array}
\right.
\end{equation}
The explicit time discretization of \eqref{ADMMdysyM1-equiv} leads to a numerical algorithm, which, for a starting point $(x^0, z^0, y^0) \in {\cal H} \times {\cal G} \times {\cal G}$, generates the sequence $(x^k, z^k, y^k)_{k \geq 0}$ for every $k \geq 0$ as follows
\begin{equation}\label{ADMMdysyM1-discr}
\left\{
\begin{array}{llll}
x^{k+1}=\prox\nolimits_{\tau_k f}\left(x^k-\tau_kA^*(y^k+cAx^k-cz^k)\right)\\
\\
y^{k+1}=\prox\nolimits_{cg^*}(y^k+cAx^{k+1})\\
\\
y^{k+1}=y^k+c (Ax^{k+1}-z^{k+1}).
\end{array}
\right.
\end{equation}
By substituting in the first equation of  \eqref{ADMMdysyM1-discr} the term $cAx^k-cz^k$ by $y^k-y^{k-1}$, which is allowed according to the last equation, one can easily see that \eqref{ADMMdysyM1-discr} is equivalent to the following numerical algorithm, which, 
for a starting point $(x^0, y^0, y^{-1}) \in {\cal H} \times {\cal G} \times {\cal G}, y^0 = y^{-1}$, generates the sequence $(x^k, y^k)_{k \geq 0}$ for every $k \geq 0$ as follows
\begin{equation}\label{ADMMdysyM1-discr1}
\left\{
\begin{array}{llll}
x^{k+1}=\prox\nolimits_{\tau_k f}\left(x^k-\tau_kA^*(2y^k-y^{k-1})\right)\\
\\
y^{k+1}=\prox\nolimits_{cg^*}(y^k+cAx^{k+1}).
\end{array}
\right.
\end{equation}
For $\tau_k=\tau >0$ for every $k\geq 0$, \eqref{ADMMdysyM1-discr1} is nothing else than the primal-dual algorithm proposed by Chambolle and Pock in \cite{ch-pck}.
\end{remark}

\begin{remark}\label{variableM}
The maps $M_1:[0,+\infty)\To S_+(\mathcal{H})$ and $M_2:[0,+\infty)\To S_+(\mathcal{G})$ can be seen as inducing a variable renorming of the  underlying Hilbert space and, as seen in the remark above, allow the use of variable step sizes. In addition, they might provide favourable settings for the derivation of convergence rates for function values along the trajectories, as it is the case for the discrete time counterpart of \eqref{ADMMdysy-subdiff} (see \cite[Section 3]{b-c-acm}).
\end{remark}

\begin{figure}[tb]
	\centering
	\captionsetup[subfigure]{position=top}
	{\includegraphics*[viewport=80 190  550 600,width=0.32\textwidth]{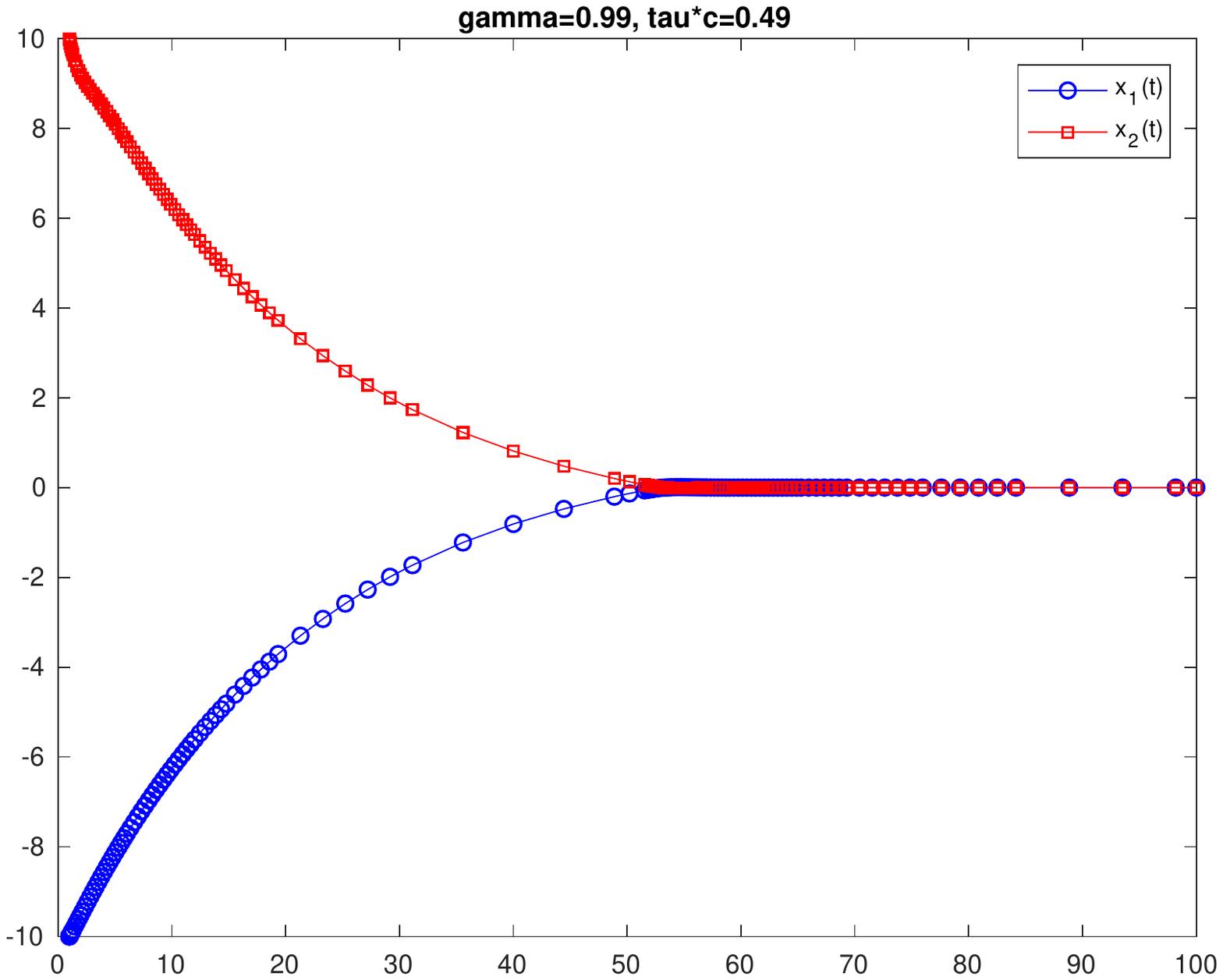}} \hspace{0.1cm}
	{\includegraphics*[viewport=80 190  550 600,width=0.32\textwidth]{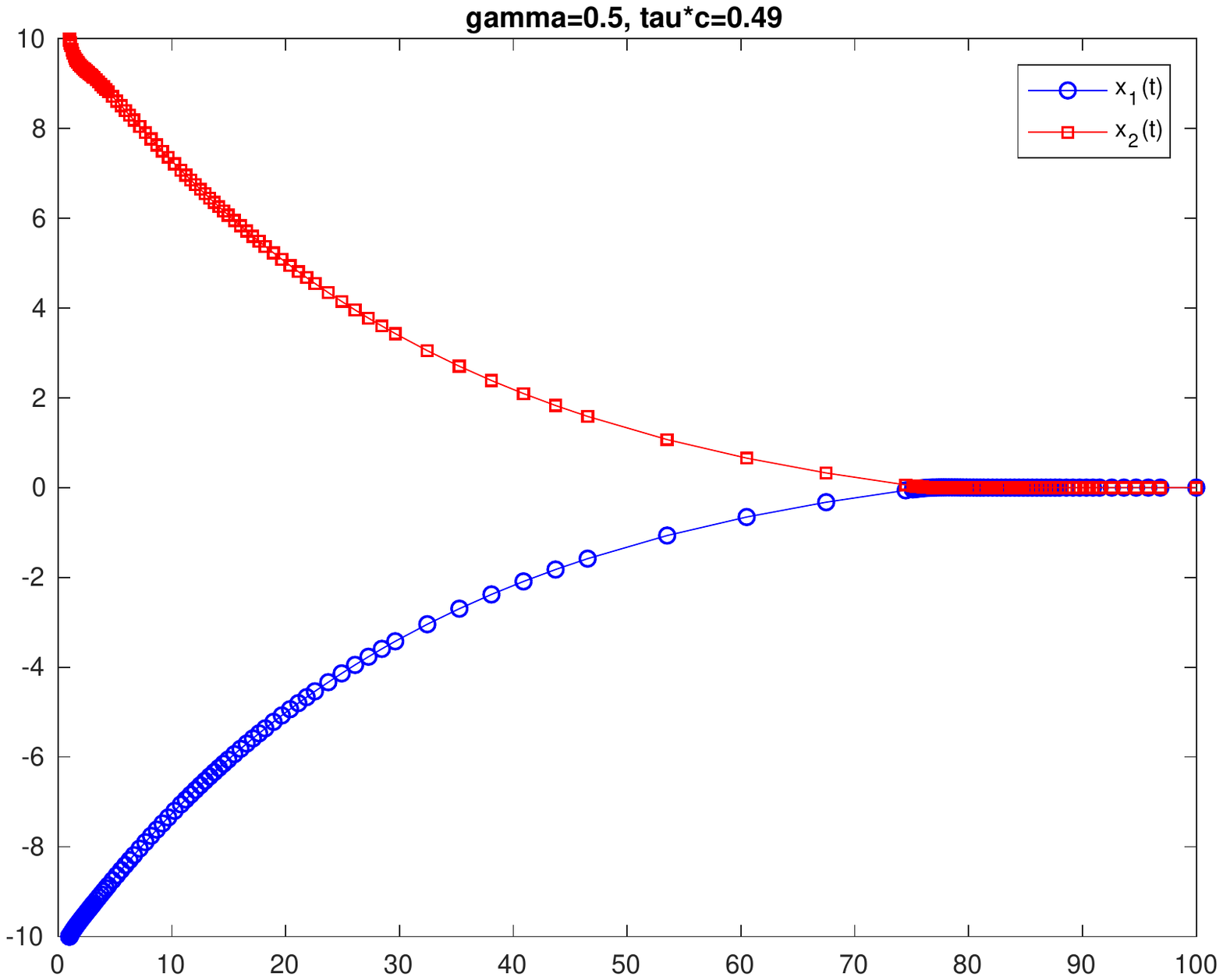}}\hspace{0.1cm}
	{\includegraphics*[viewport=80 190  550 600,width=0.32\textwidth]{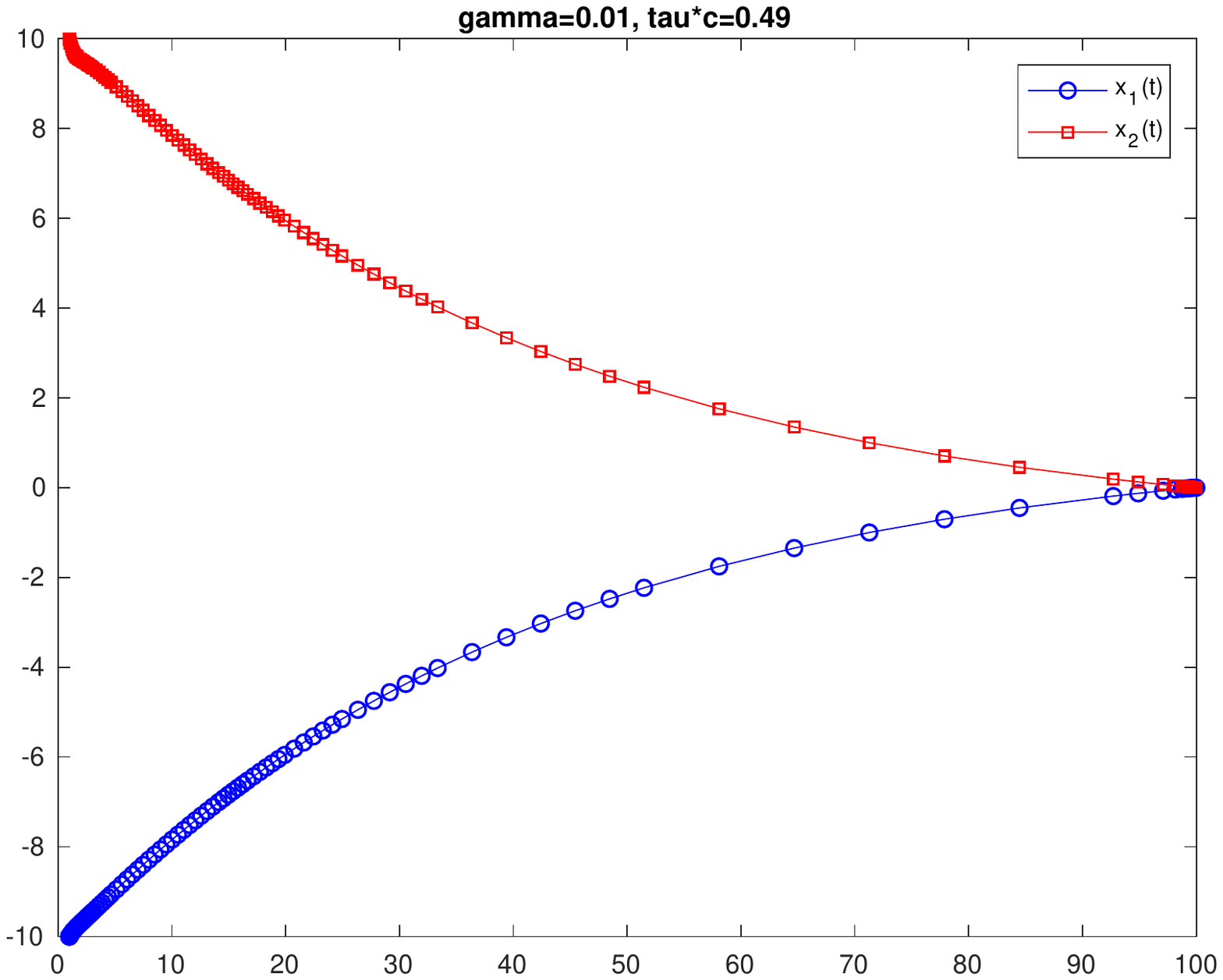}}\\
	{\includegraphics*[viewport=80 190  550 600,width=0.32\textwidth]{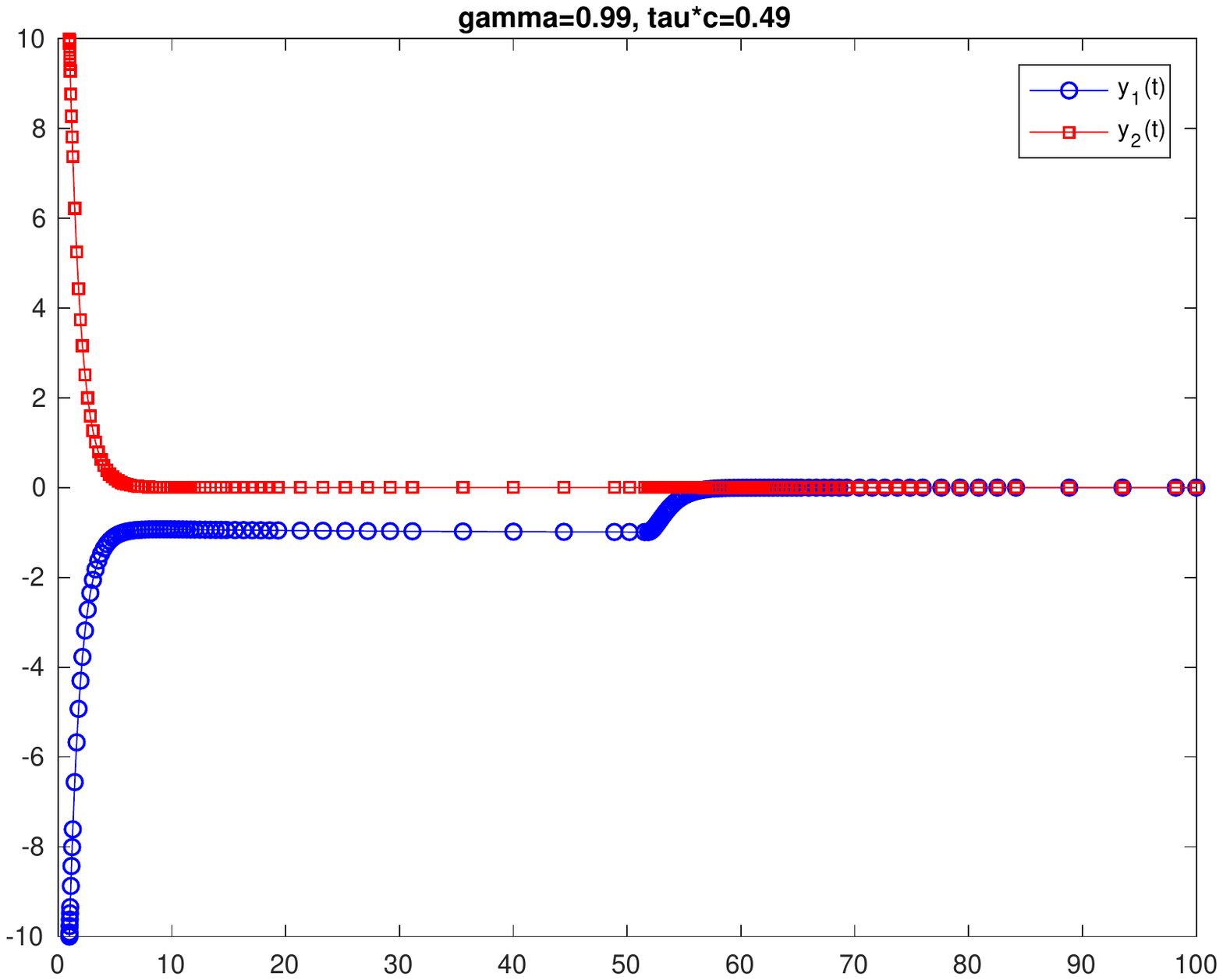}} \hspace{0.1cm}
	{\includegraphics*[viewport=80 190  550 600,width=0.32\textwidth]{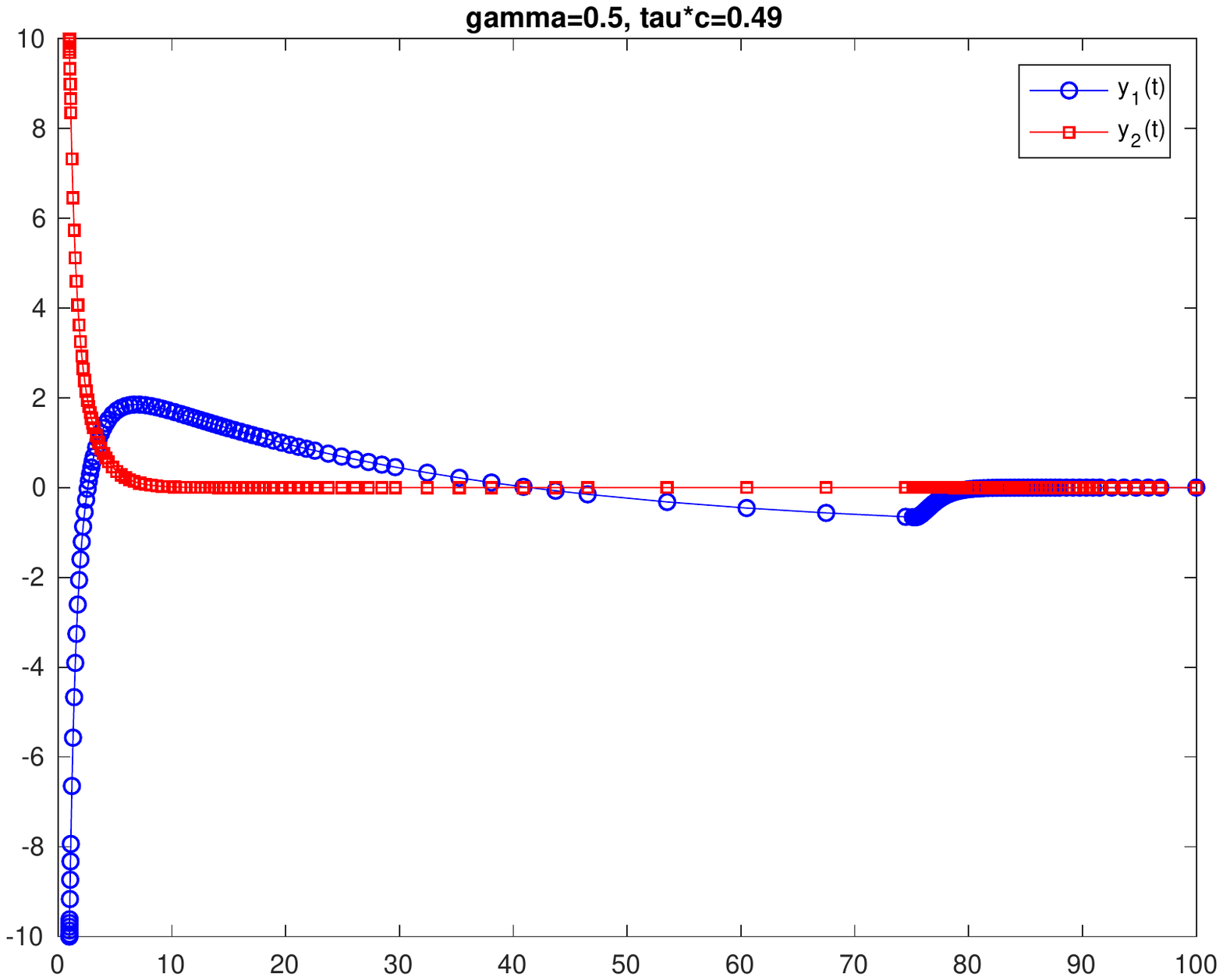}}\hspace{0.1cm}
	{\includegraphics*[viewport=80 190  550 600,width=0.32\textwidth]{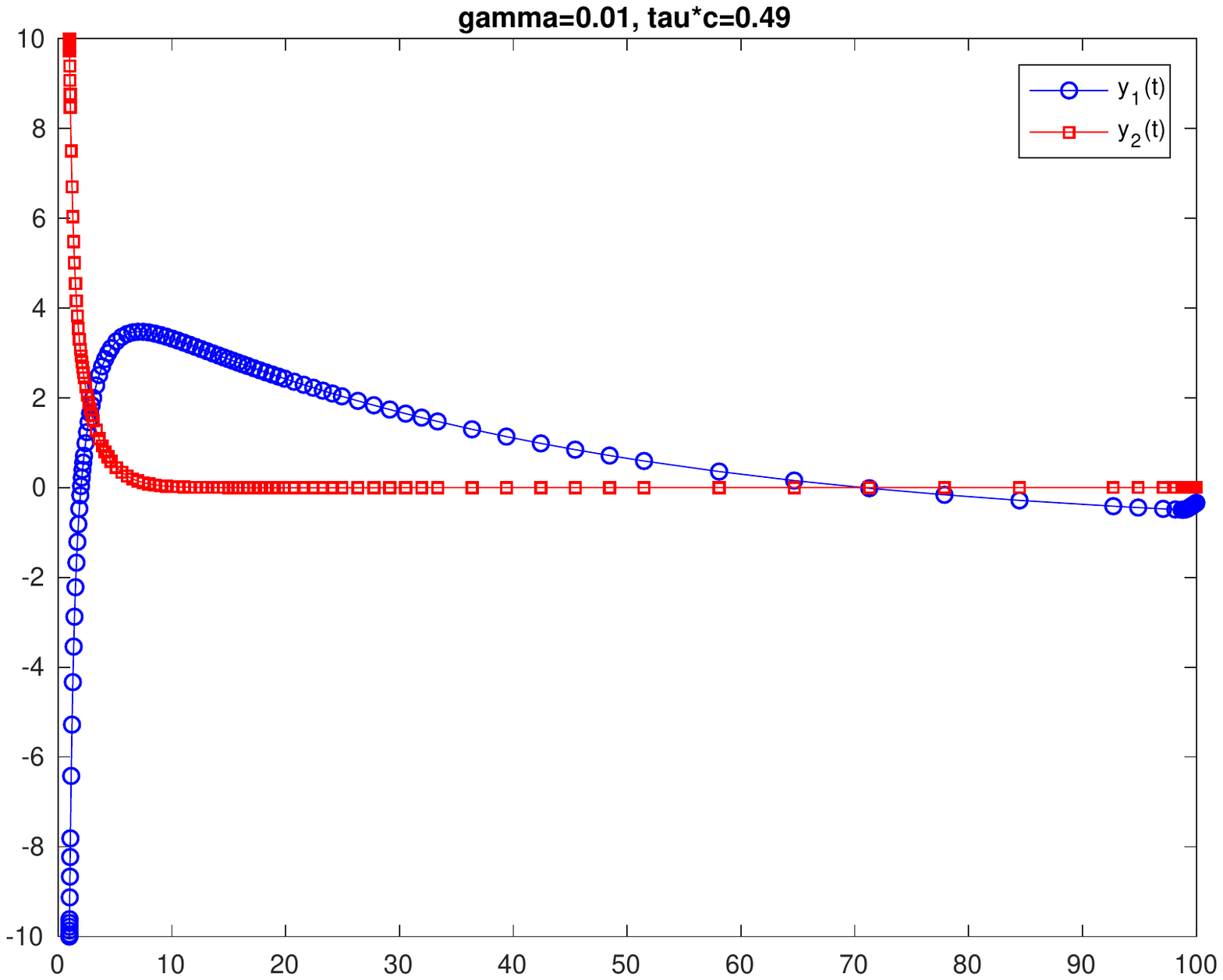}}
	\caption{\small First row: the primal trajectory $x(t)$ approaching the primal optimal solution $(0,0)$ for $\tau c = 0.49$ and starting point $x^0=(-10,10)$. Second row: the dual trajectory $y(t)$ approaching a dual optimal solution for $\tau c = 0.49$ and starting point $y^0=(-10,10)$.}
	\label{fig:ex1}	
\end{figure}

\begin{example} \label{example1} 
We will illustrate the way in which the parameters $\gamma, c$ and $\tau(t), t \in [0,+\infty)$ may influence the asymptotic convergence
of the primal and dual trajectories via numerical experiments. In this scope, we considered the following primal optimization problem
\begin{equation}\label{primalex}
\inf_{(x_1,x_2) \in \R^2} \frac{1}{2}(x_1^2 + x_2^2) + |x_1-x_2| + |x_1+x_2|,
\end{equation}
which is in fact problem \eqref{primal} written in the following particular setting:  ${\mathcal H}={\mathcal G}=\R^2$, $f,g, h:\R^2\to \R$, $f(x)= \frac{1}{2}\|x\|_2^2$,
$g(x)=\|x\|_1$, $h(x)=0$, for every $x \in \R^2$, and $A:\R^2\to\R^2$, $A(x_1,x_2)=(x_1-x_2,x_1+x_2)$. Then $\ol x=(0,0)$ is the unique optimal solution of
\eqref{primalex} and that 
\begin{equation}\label{dualex}
\sup_{\|(y_1,y_2)\|_\infty \leq 1} - y_1^2 - y_2^2
\end{equation}
is the Fenchel dual problem of \eqref{primalex}. Thus $\ol y = (0,0)$ is the unique dual optimal solution.

\begin{figure}[tb]
	\centering
	\captionsetup[subfigure]{position=top}
	{\includegraphics*[viewport=80 190  550 600,width=0.32\textwidth]{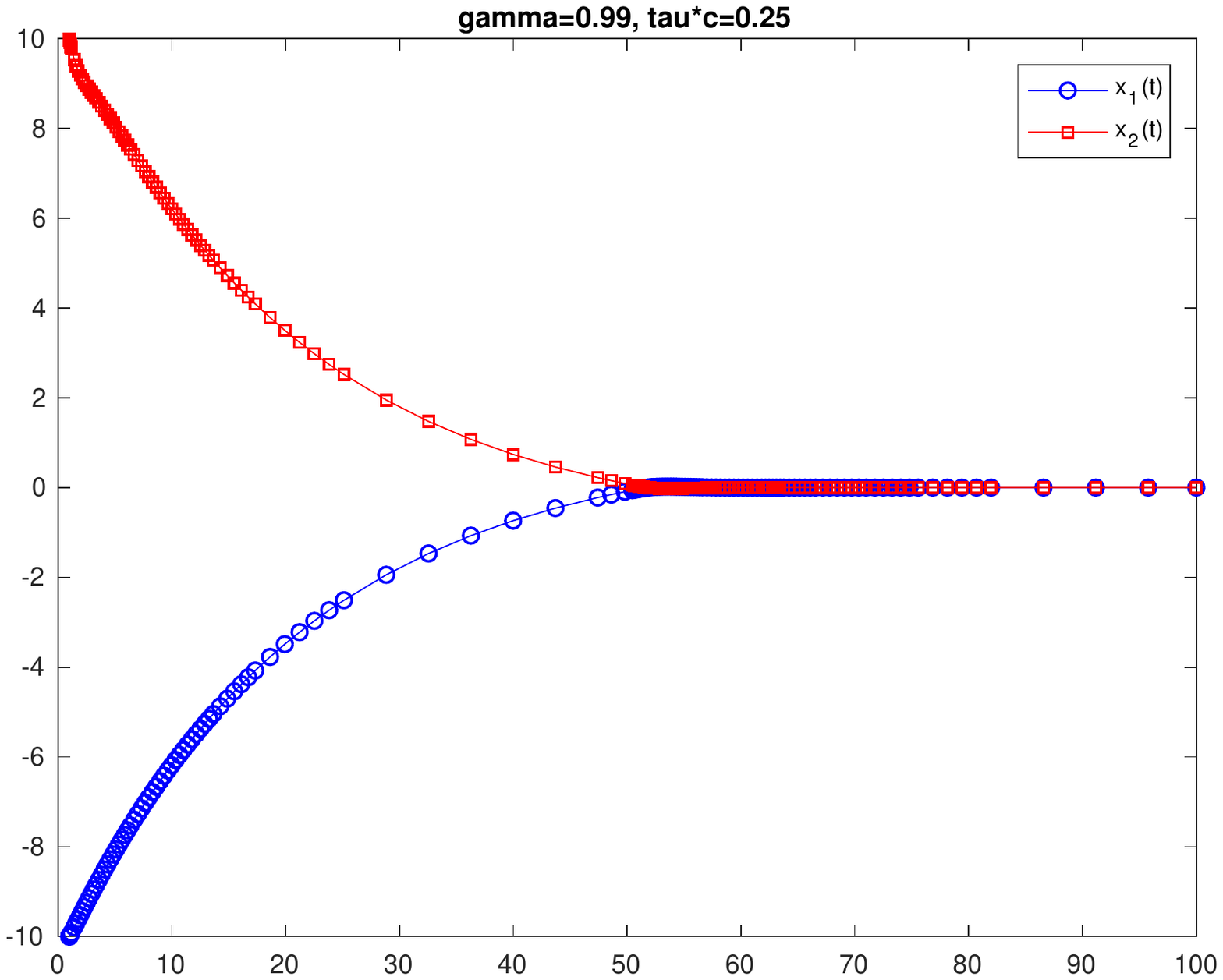}} \hspace{0.1cm}
	{\includegraphics*[viewport=80 190  550 600,width=0.32\textwidth]{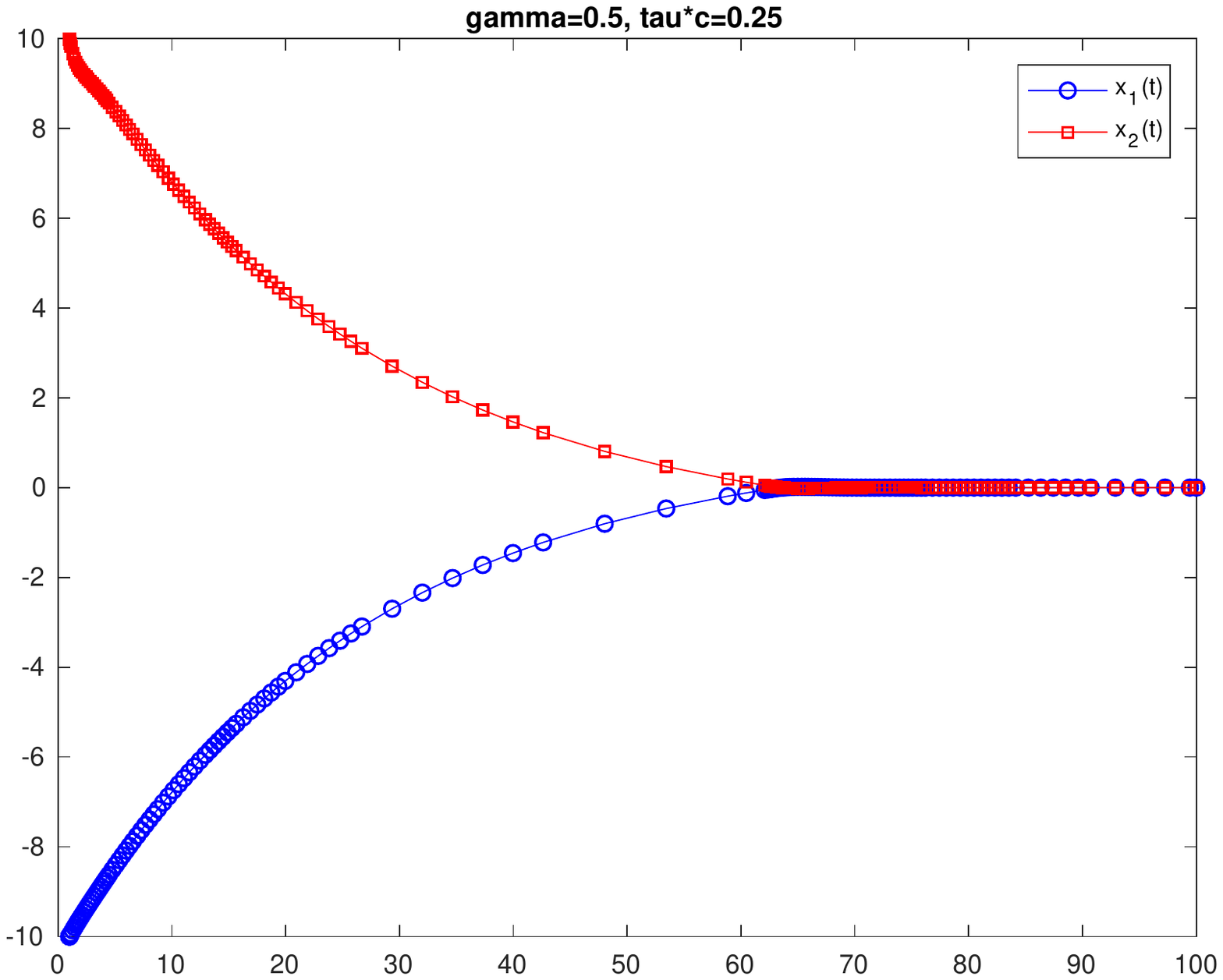}}\hspace{0.1cm}
	{\includegraphics*[viewport=80 190  550 600,width=0.32\textwidth]{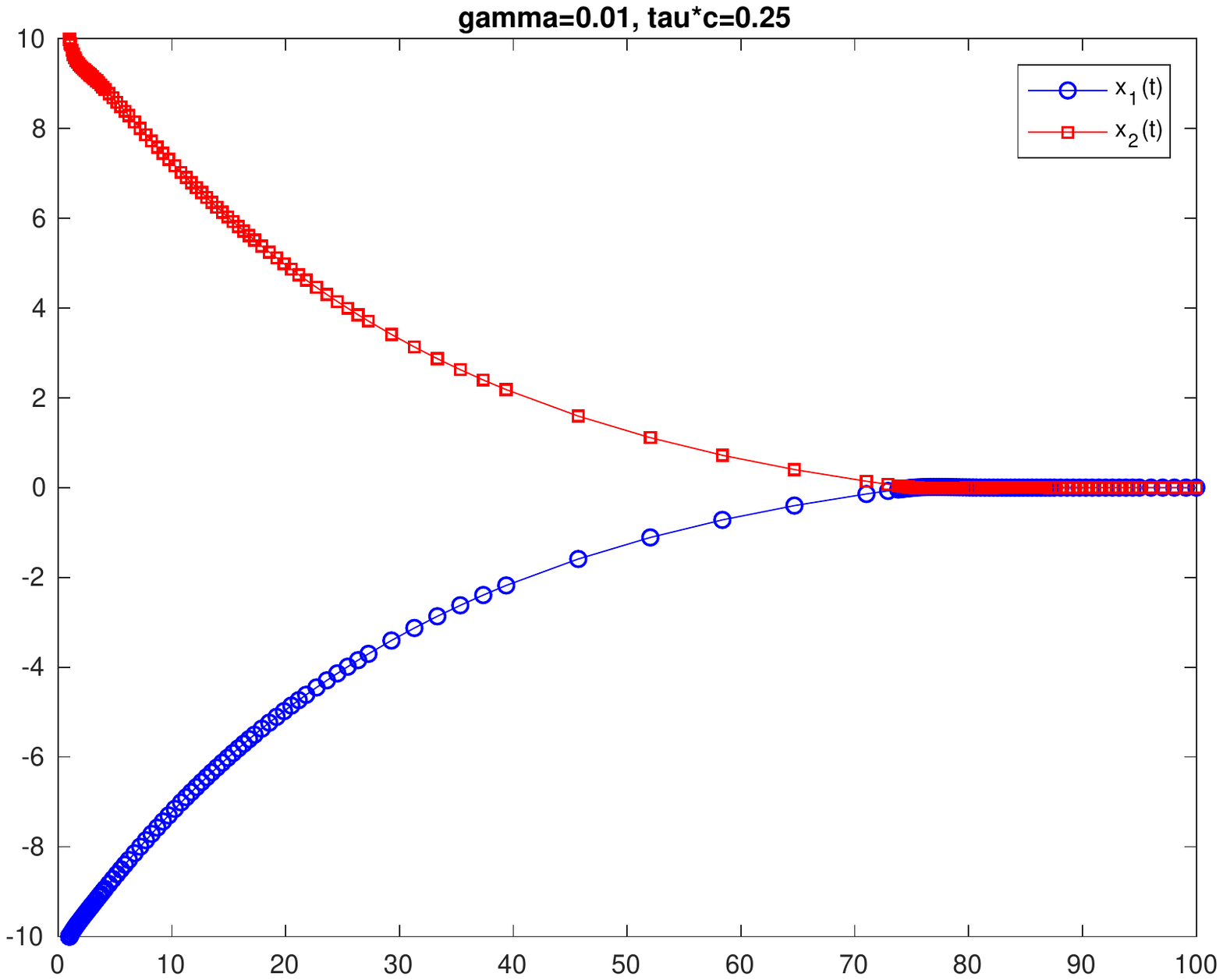}}\\
	{\includegraphics*[viewport=80 190  550 600,width=0.32\textwidth]{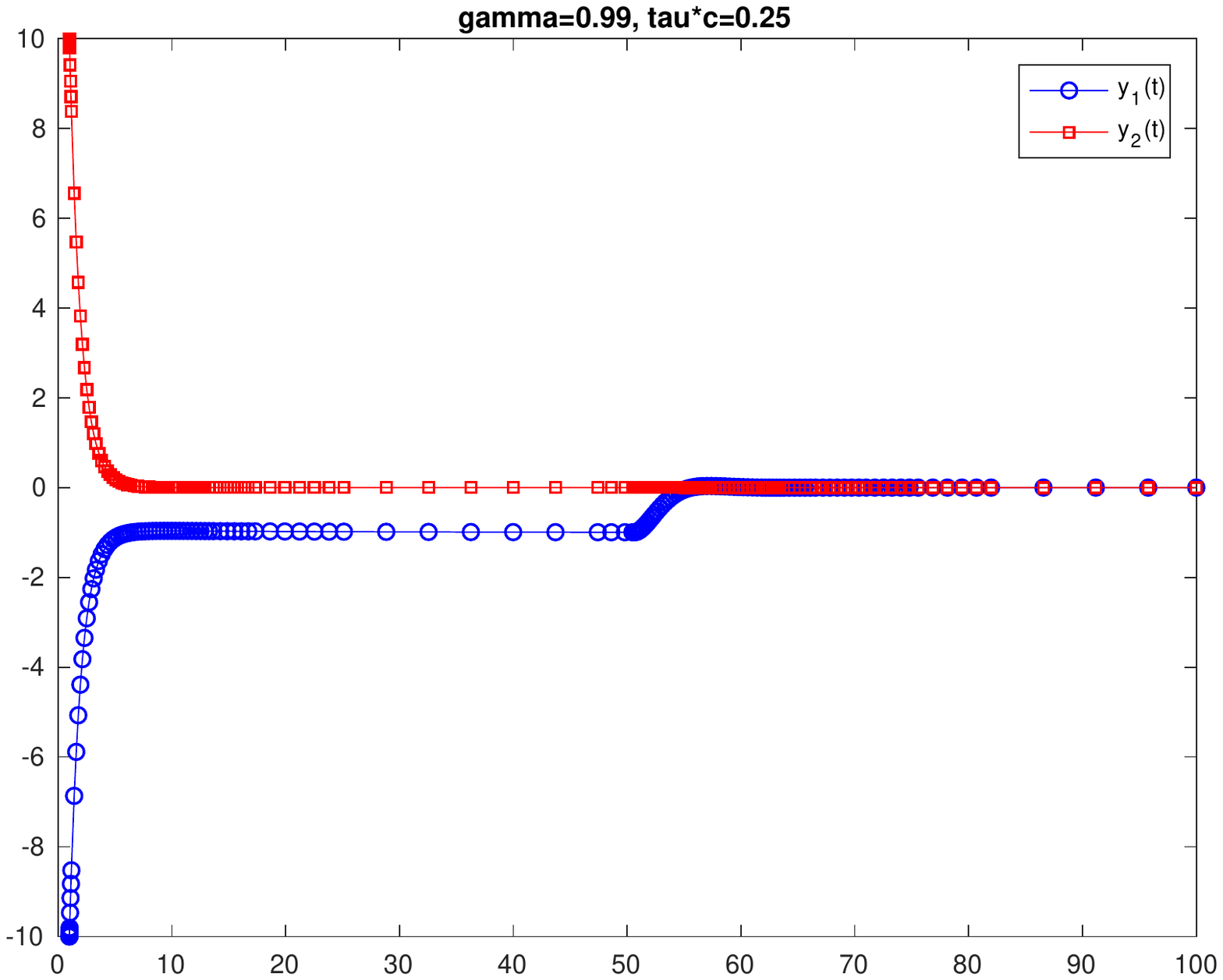}} \hspace{0.1cm}
	{\includegraphics*[viewport=80 190  550 600,width=0.32\textwidth]{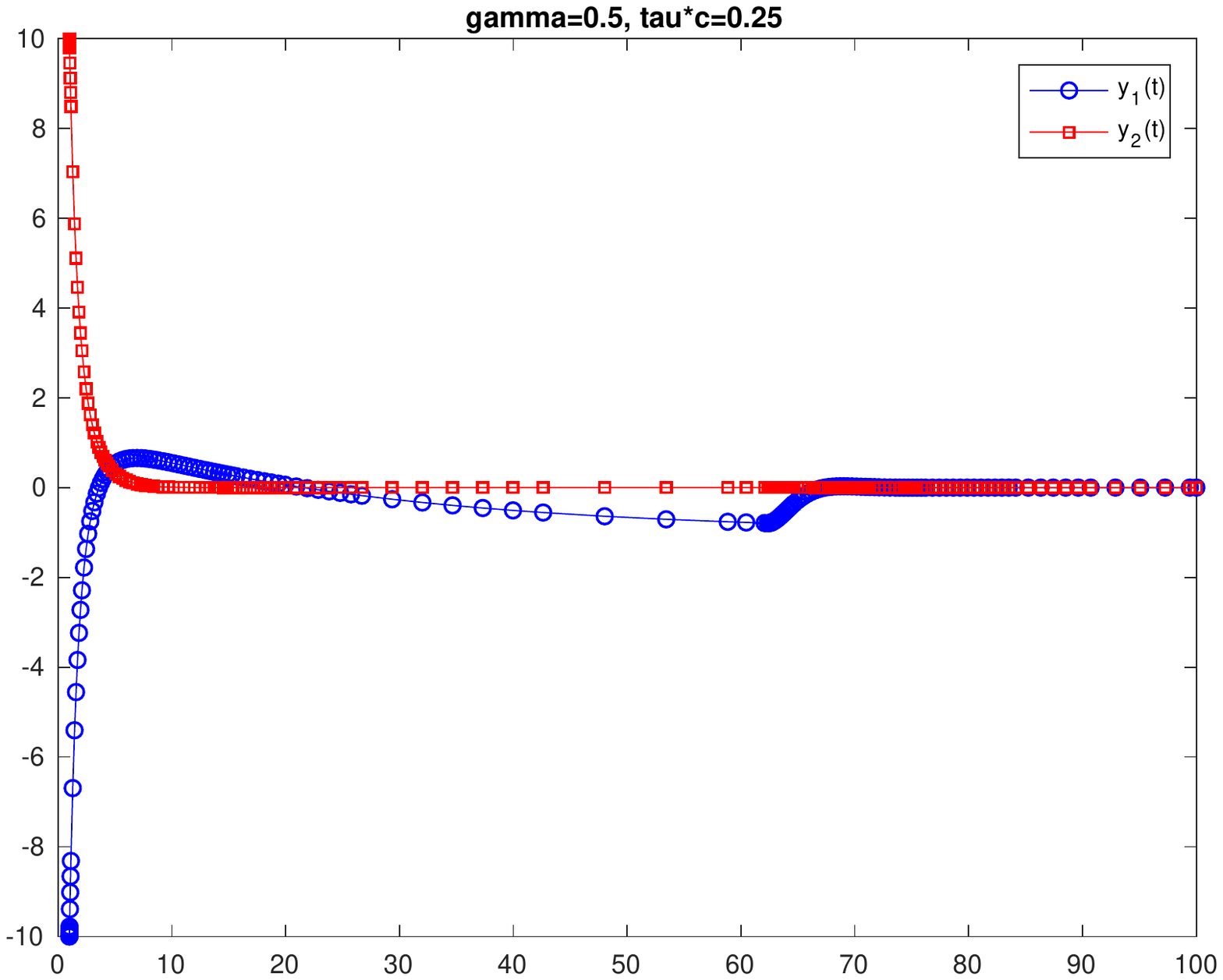}}\hspace{0.1cm}
	{\includegraphics*[viewport=80 190  550 600,width=0.32\textwidth]{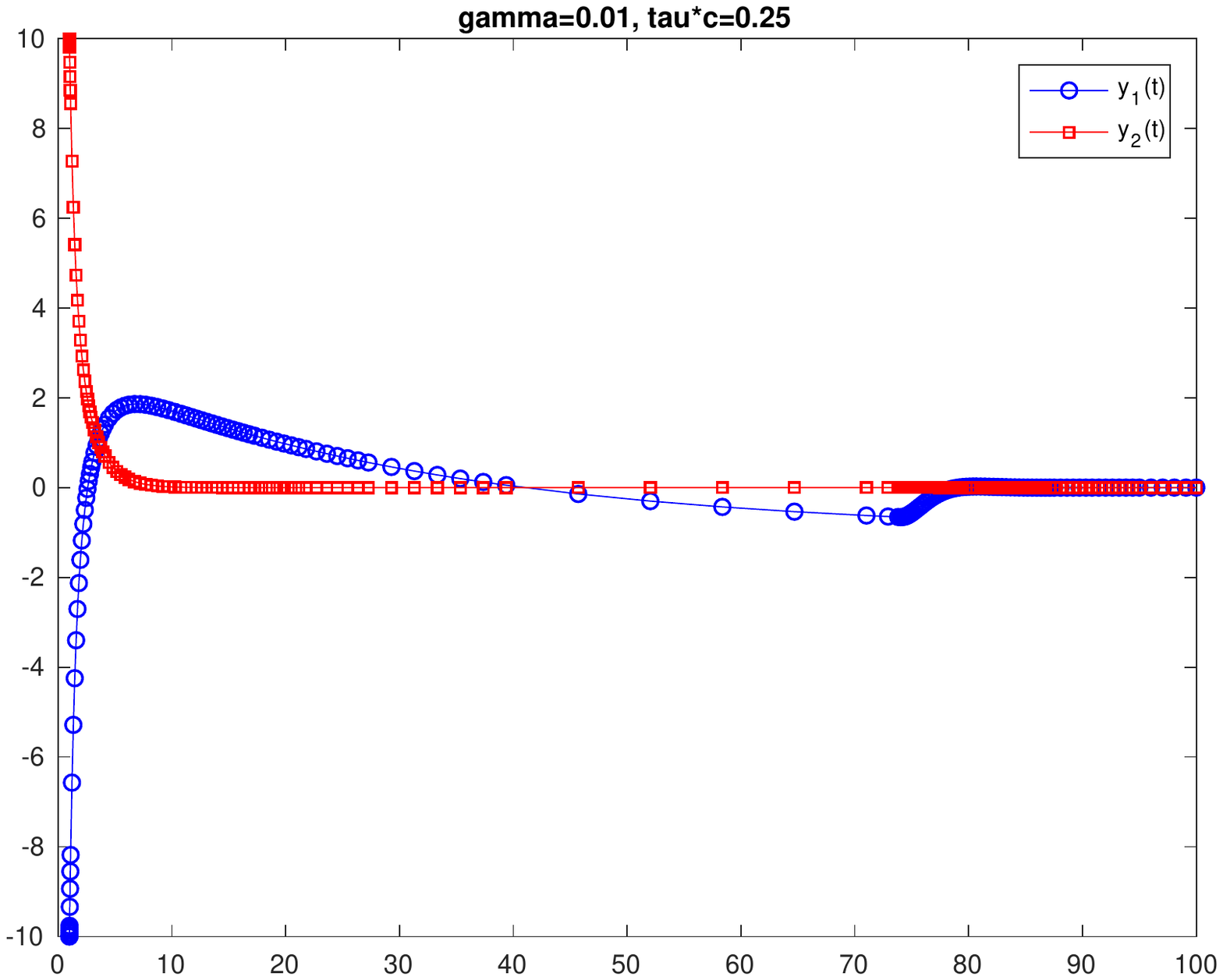}}
	\caption{\small First row: the primal trajectory $x(t)$ approaching the primal optimal solution $(0,0)$ for $\tau c = 0.25$ and starting point $x^0=(-10,10)$. Second row: the dual trajectory $y(t)$ approaching a dual optimal solution for $\tau c = 0.25$ and starting point $y^0=(-10,10)$.}
	\label{fig:ex2}	
\end{figure}

We considered the dynamical system \eqref{ADMMdysyM1} attached to the primal-dual pair \eqref{primalex}-\eqref{dualex} with starting points $x^0=(-10,10)$, $z^0 = Ax^0 = (-20,0)$ and $y^0=(-10,10)$
in the case when $\tau(t) = \tau >0$ for every $t \in [0,+\infty)$ is a constant function.
In order to solve the resulting dynamical system we used the Matlab function {\tt ode15s} and, to this end, we reformulated it as
\begin{equation*}
\left \{
\begin{array}{l}
\dot{U}(t) =\G(U(t))\\
U(0) =(x^0,y^0,z^0),
\end{array}
\right.
\end{equation*}
where 
$$U(t)=(x(t),y(t),z(t))\in {\mathcal H}\times {\mathcal G}\times {\mathcal G}$$ 
and 
$$\Gamma:{\mathcal H}\times {\mathcal G}\times {\mathcal G}\to {\mathcal H}\times {\mathcal G}\times {\mathcal G}, \ \Gamma (u_1,u_2,u_3)=(u_4,u_5,u_6),$$
is defined as 
\begin{equation*}
\left\{
\begin{array}{llll}
u_4=\prox\nolimits_{\tau f}\left(u_1-\tau A^*(u_2+cAu_1-cu_3)\right)-u_1\\
\\
u_5=\prox\nolimits_{cg^*}(u_2+cA(\gamma u_4+u_1))-u_2-c(\gamma -1)Au_4\\
\\
u_6=A(u_1+u_4)-u_3-\frac{1}{c}u_5.
\end{array}
\right.
\end{equation*}
Notice that
\begin{align*}
A^*(x_1,x_2) & =  (x_1+x_2,-x_1+x_2) \quad \forall (x_1, x_2) \in \R^2,\\
\prox\nolimits_{\tau f}(x) & =  \frac{1}{1+\tau} x \quad \forall x \in \R^2, \\
\prox\nolimits_{c g^*}(y) & = \proj\nolimits_{[-1,1]^2}(y) \quad \forall y \in \R^2,
 \end{align*}
where $\proj_Q$ denotes the projection operator on a convex and closed set $Q \subseteq {\cal H}$.

According to Theorem \ref{convergenceM1}, the asymptotic convergence of the trajectories can be guaranteed when
$\tau c \|A\|^2 \leq 1$. Since $\|A\| = \sqrt{2}$, we considered for $\tau c \in (0, \frac{1}{2})$ three different choices, namely, $\tau c = 0.49, 0.25$ and $0.01$.
The primal and the dual trajectories generated by the dynamical system for each of these three choices are represented in the figures \ref{fig:ex1}, \ref{fig:ex2} and
\ref{fig:ex3}, respectively. The first row of each figure represents the primal trajectories $x(t)$ for $\gamma = 0.99, 0.5$ and $0.01$, while the second row represents the dual trajectories $y(t)$ for the same choices of the parameter $\gamma$. 

One can see that the parameter $\gamma$ has a strong influence on the asymptotic behaviour of the primal and dual trajectories. Namely, in all three figures, thus independently of the choice of the parameters $\tau$ and $c$, 
the primal and the dual trajectories converge faster to the corresponding primal and dual solutions, respectively, for larger values of $\gamma$, namely, when $\gamma$ is closer to $1$. As $\gamma$ is the coefficient of the derivative in the second inclusion of the dynamical system, it can be seen as a constant which characterizes its level of implicitness. In particular, more implicitness promotes a better asymptotic convergence. On the other hand, we notice that the smaller the values of $\tau c$ are, the smaller is the influence of $\gamma$ on the asymptotic convergence of the trajectories.

\begin{figure}[tb]
	\centering
	\captionsetup[subfigure]{position=top}
	{\includegraphics*[viewport=80 190  550 600,width=0.32\textwidth]{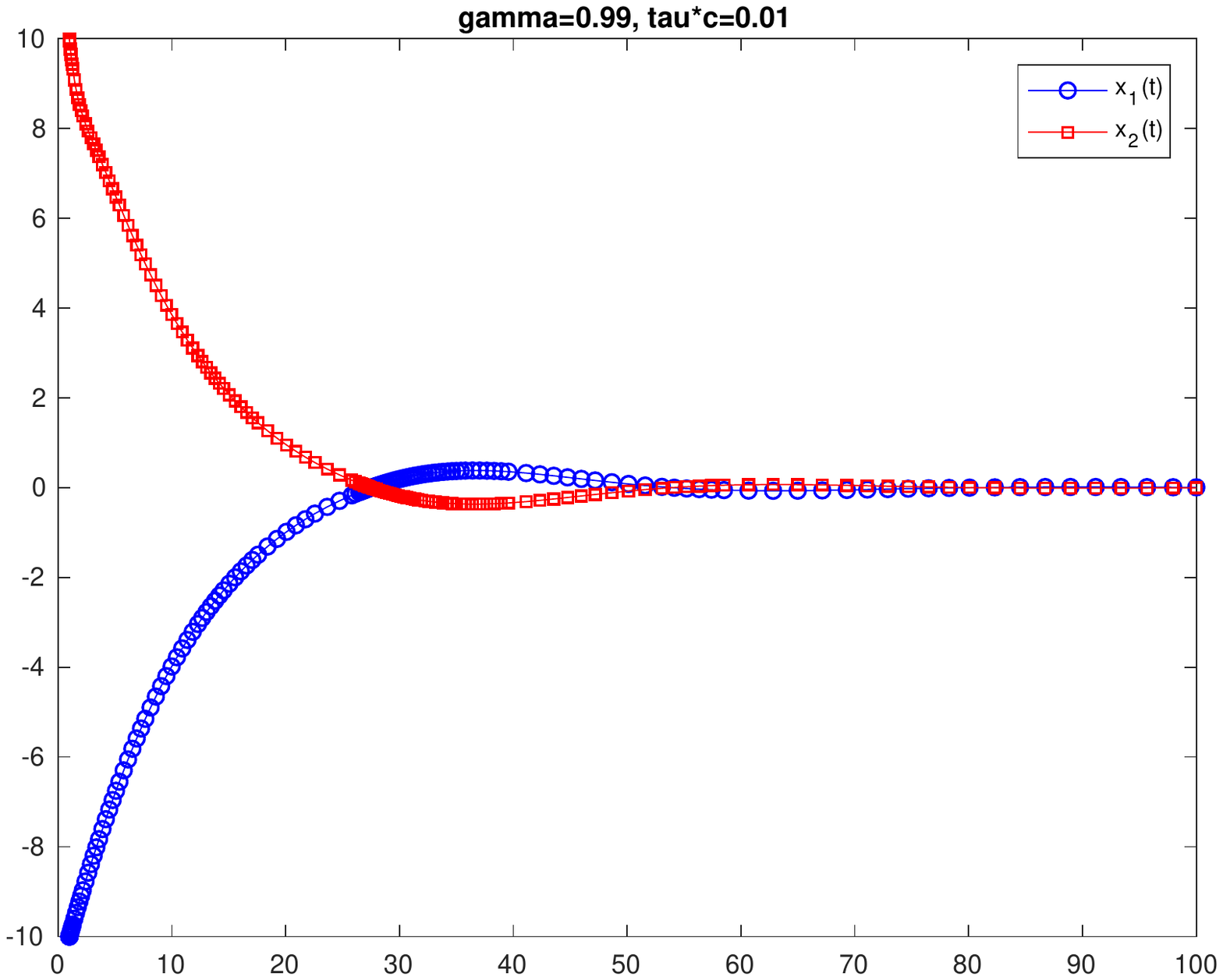}} \hspace{0.1cm}
	{\includegraphics*[viewport=80 190  550 600,width=0.32\textwidth]{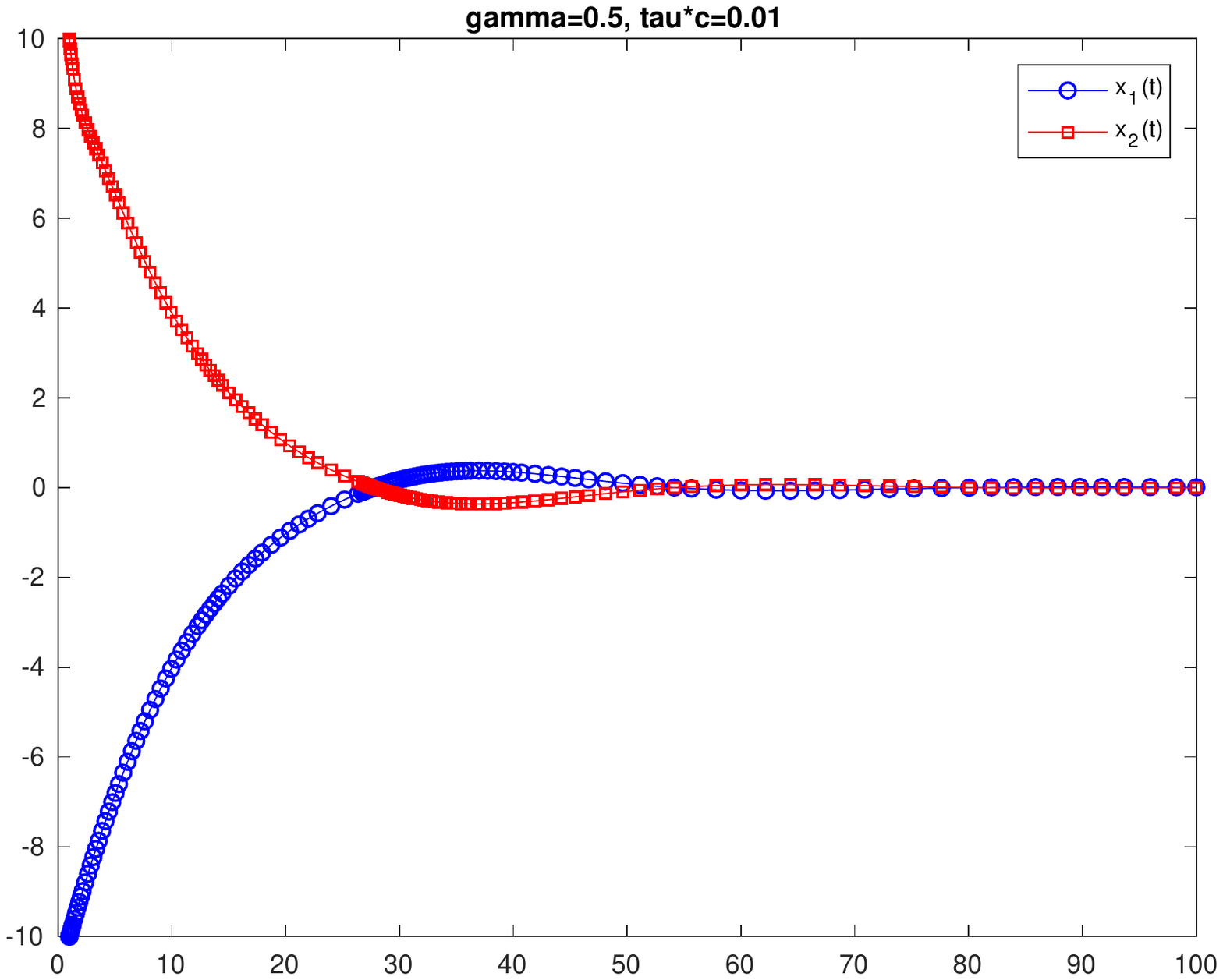}}\hspace{0.1cm}
	{\includegraphics*[viewport=80 190  550 600,width=0.32\textwidth]{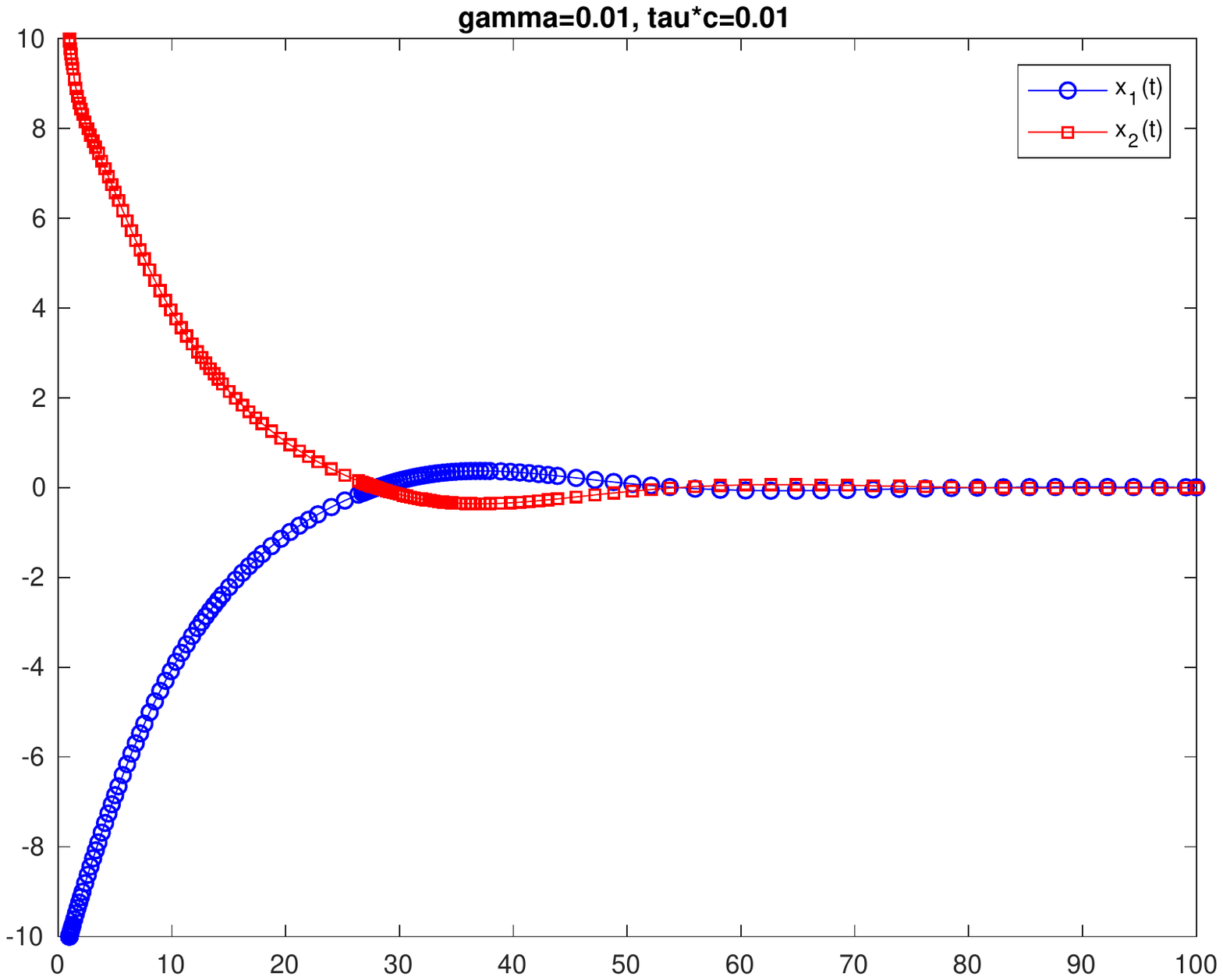}}\\
	{\includegraphics*[viewport=80 190  550 600,width=0.32\textwidth]{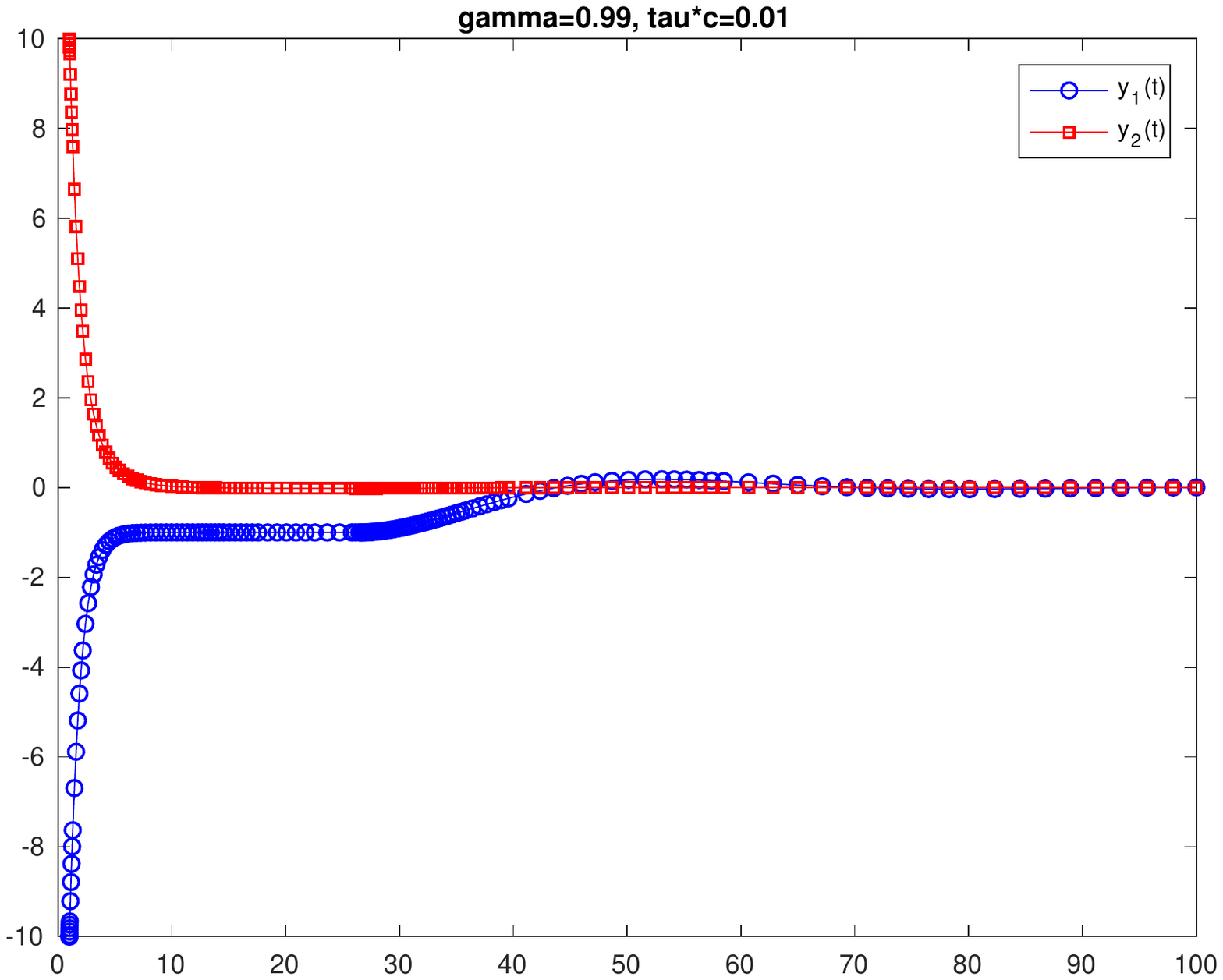}} \hspace{0.1cm}
	{\includegraphics*[viewport=80 190  550 600,width=0.32\textwidth]{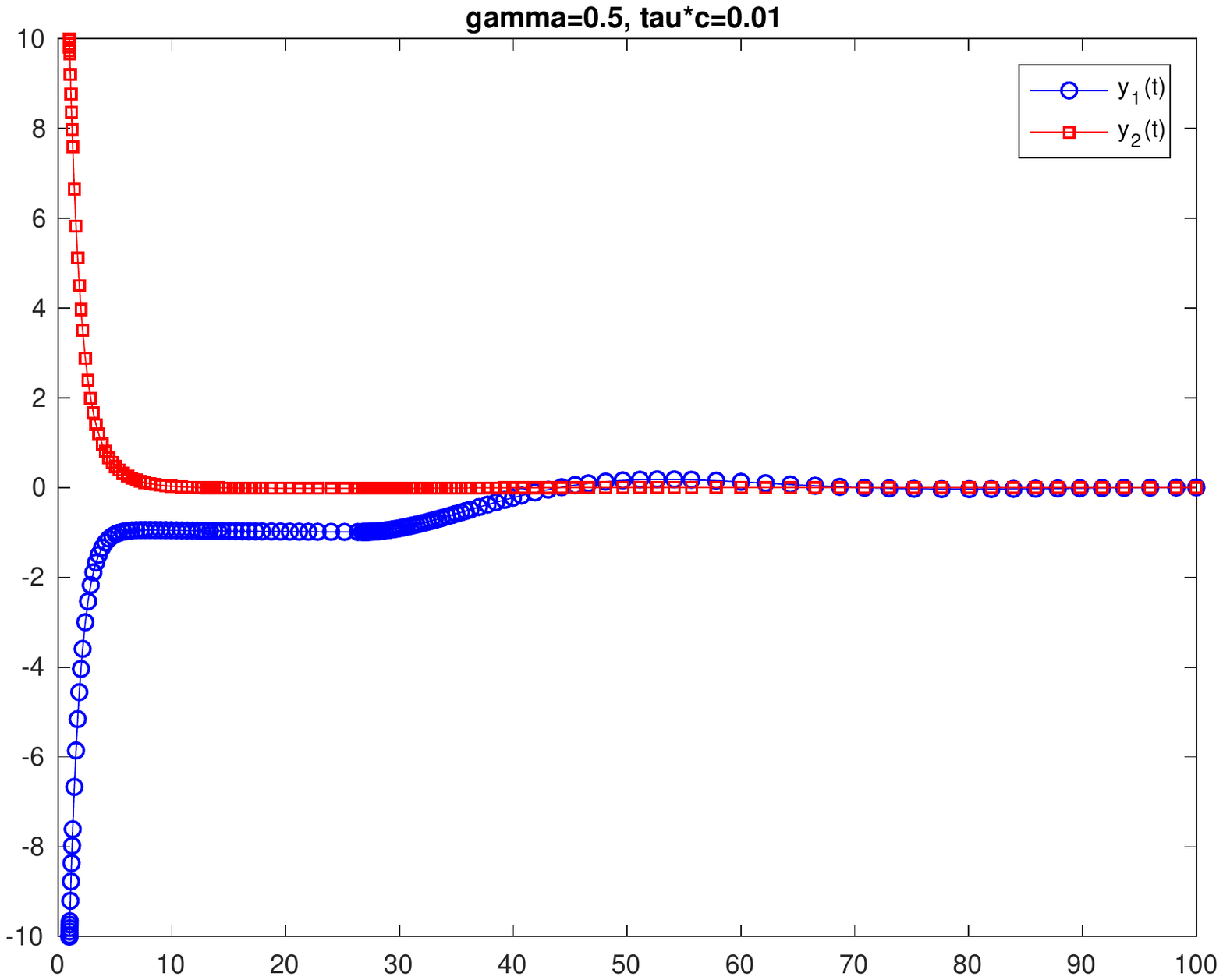}}\hspace{0.1cm}
	{\includegraphics*[viewport=80 190  550 600,width=0.32\textwidth]{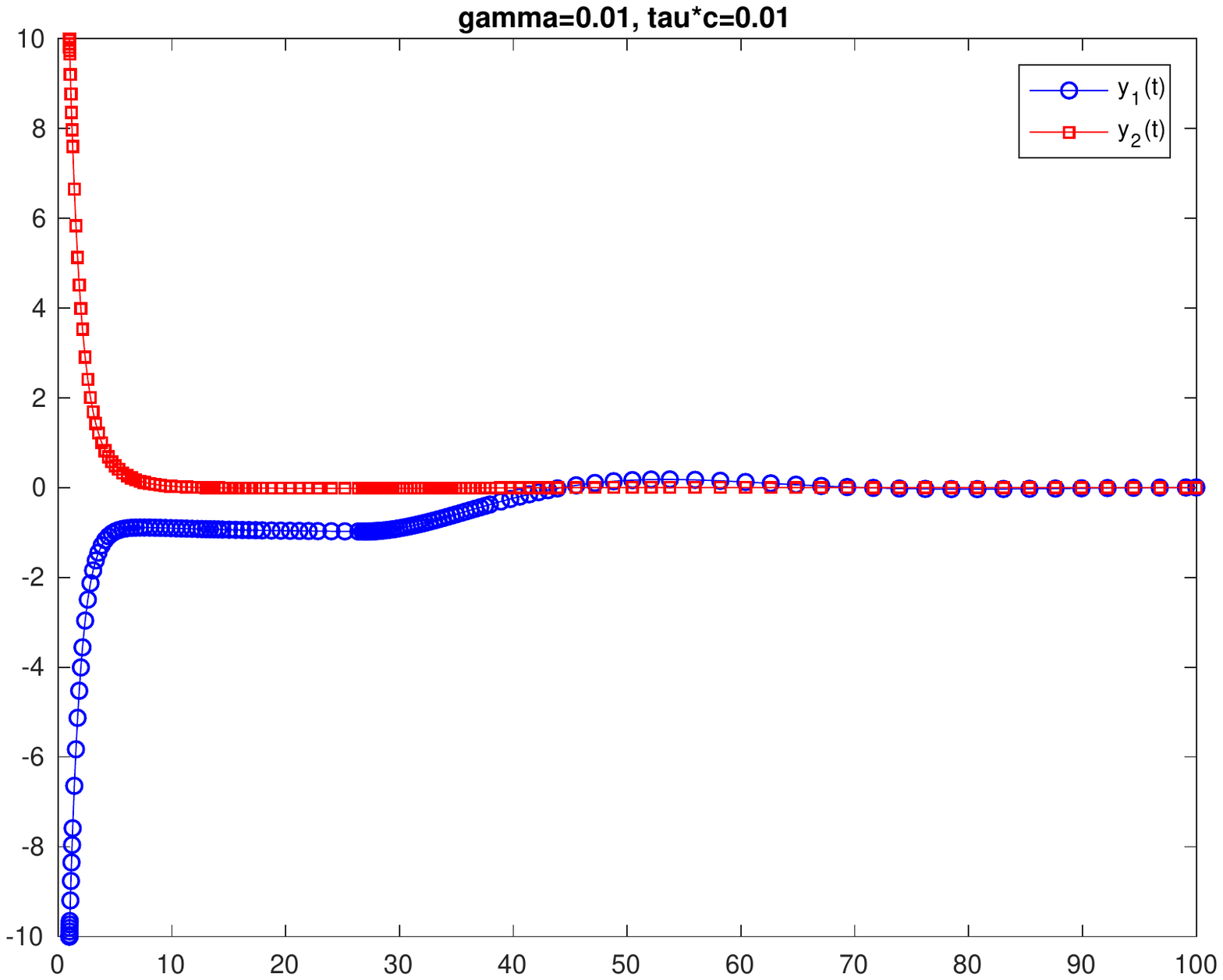}}
	\caption{\small First row: the primal trajectory $x(t)$ approaching the primal optimal solution $(0,0)$ for $\tau c = 0.1$ and starting point $x^0=(-10,10)$. Second row: the dual trajectory $y(t)$ approaching a dual optimal solution for $\tau c = 0.1$ and starting point $y^0=(-10,10)$.}
	\label{fig:ex3}	
\end{figure}
\end{example}

\begin{notation}\label{notation}
The following two functions will play an important role in particular in the forthcoming analysis 
$$F:[0,+\infty)\times\mathcal{H}\To\oR, \quad F(t,x)=f(x)+\frac{c}{2}(\|Ax\|^2-\|x\|^2)+\frac12\|x\|^2_{M_1(t)},$$
and 
$$G:[0,+\infty)\times\mathcal{G}\To\oR, \quad G(t,z)=g(z)+\frac12\|z\|^2_{M_2(t)}.$$
With these two notations, the dynamical system \eqref{ADMMdysy-subdiff} can be rewritten as
\begin{equation}\label{ADMMdysy}
\left\{
\begin{array}{llll}
\dot{x}(t)+x(t)\in\argmin\limits_{x\in\mathcal{H}}\left(F(t,x)+\frac{c}{2}\left\|x-\left(\frac{1}{c}M_1(t)x(t)+A^*z(t)-\frac{A^*}{c}y(t)-\frac{1}{c}\n h(x(t))\right)\right\|^2\right)\\
\\
\dot{z}(t)+z(t)=\argmin\limits_{z\in\mathcal{G}}\left(G(t,z)+\frac{c}{2}\left\|z-\left(\frac{1}{c}M_2(t)z(t)+A(\g\dot{x}(t)+x(t))+\frac{1}{c}y(t)\right)\right\|^2\right)\\
\\
\dot{y}(t)=c A(x(t)+\dot{x}(t))-c (z(t)+\dot{z}(t))\\
\\
x(0)=x^0\in{\mathcal H},\,y(0)=y^0\in{\mathcal G},\,z(0)=z^0\in{\mathcal G}.
\end{array}
\right.
\end{equation}
Let $t \in [0,+\infty)$ be fixed. The function $G(t,\cdot)$ is proper, convex and lower semicontinuous, hence $z\rightarrow G(t,z)+\frac{c}{2}\|z-v\|^2$ is  proper, strongly convex and lower semicontinuous for every $v \in {\cal G}$. This allows us to use the sign equal 
in the second relation of \eqref{ADMMdysy}. On the other hand, a sufficient condition which guarantees that the function $x \mapsto F(x,t) + \frac{c}{2}\left\|x-\left(\frac{1}{c}M_1(t)x(t)+A^*z(t)-\frac{A^*}{c}y(t)-\frac{1}{c}\n h(x(t))\right)\right\|^2$, which is proper and lower semicontinuous, is strongly convex is that
there exists $\alpha(t) >0$ such that $cA^*A+M_1(t)\in P_{\a(t)}(\mathcal{H})$. This actually ensures that $x\rightarrow F(t,x)+\frac{c}{2}\|x-u\|^2$ is proper, strongly convex and lower semicontinuous for every $u \in {\cal H}$.

This means that if the assumption 
$$(Cweak) \quad \text{for every} \  t\in[0,+\infty) \ \text{there exists} \ \a(t)>0 \ \text{such that} \ cA^*A+M_1(t)\in P_{\a(t)}(\mathcal{H})$$
holds, then we can use also in the first relation of \eqref{ADMMdysy} the sign equal. It is easy to see, that, if $(Cweak)$ holds, then $\partial f + cA^*A + M_1(t)$ is $\alpha(t)$-strongly monotone for every $t \in [0,+\infty)$. In other words, for every $t \in [0,+\infty)$, all $u,v \in {\cal H}$ and all $u^* \in (\partial f + cA^*A + M_1(t))(u), x^* \in (\partial f + cA^*A + M_1(t))(x)$ we have
$$\left \langle u^* - x^*, u-x \right \rangle \geq \alpha(t) \|u-x\|^2.$$
Notice that, since $A^*A\in S_+(\mathcal{H})$ and $M_1(t)\in S_+(\mathcal{H})$ for every $t\in[0,+\infty)$, $(Cweak)$ is fulfilled, if
\begin{equation}\label{C1}
\text{for every} \  t\in[0,+\infty) \ \text{there exists} \ \a(t)>0 \ \text{such that} \ M_1(t)\in P_{\a(t)}(\mathcal{H})
\end{equation}
or, if
\begin{equation}\label{C2}
\text{there exists} \ \a >0 \ \text{such that} \ A^*A \in P_{\a}(\mathcal{H}).
\end{equation}
Notice also that, if ${\cal H}$ is a finite dimensional Hilbert space, then \eqref{C2}, which is independent of $t$, is nothing else than $A^*A$ is positively definite or, equivalently, $A$ is injective.

Let $S=\{x\in\mathcal{H}:\|x\|=1\}$ be the unit sphere of $\mathcal{H}$. Assumption $(Cweak)$ is fulfilled if and only if $\inf_{x\in S}\|x\|^2_{cA^*A+M_1(t)} >0$ for every $t \in [0,+\infty)$. In this case we can take
$\a(t):=\inf_{x\in S}\|x\|^2_{cA^*A+M_1(t)}$ for every $t \in [0,+\infty)$. 

Obviously, if $(Cstrong)$ holds, then $(Cweak)$ holds with $\alpha(t) := \alpha >0$ for every $t\in[0,+\infty)$.
\end{notation}

\section{Existence and uniqueness of the trajectories}\label{sec2}

In this section we will investigate the existence and uniqueness of the trajectories generated by \eqref{ADMMdysy-subdiff}. We start by recalling the definition of a locally absolutely continuous map.

\begin{definition}\label{abs-cont}
A function $x:[0,+\infty)\rightarrow {\mathcal H}$  is said to be locally absolutely continuous, if it is absolutely continuous on every interval $[0,T],\,T>0$; that is, for every $T >0$ there exists an integrable function $y:[0,T]\rightarrow {\mathcal H}$ such that $$x(t)=x(0)+\int_0^t y(s)ds \ \ \forall t\in[0,T].$$
\end{definition}

\begin{remark}\label{rem-abs-cont} (a)  Every absolutely continuous function is differentiable almost
everywhere, its derivative coincides with its distributional derivative almost everywhere and one can recover the function from its derivative $\dot x=y$ by the above integration formula.

(b) Let be $T >0$ and  $x:[0,T]\rightarrow {\mathcal H}$ an absolutely continuous function. This is equivalent to (see \cite{att-sv2011, abbas-att-sv}): for every $\varepsilon > 0$ there exists $\eta >0$ such that for any finite family of intervals $I_k=(a_k,b_k)\subseteq [0,T]$ the following property holds:
$$ \mbox{for any subfamily of disjoint intervals} \ I_j \ \mbox{with} \ \sum_j|b_j-a_j| < \eta \ \mbox{it holds} \ \sum_j\|x(b_j)-x(a_j)\| < \varepsilon.$$
From this characterization it is easy to see that, if $B:{\mathcal H}\rightarrow {\mathcal H}$ is $L$-Lipschitz continuous with $L\geq 0$, then the function $z=B\circ x$ is absolutely continuous, too.  This means that $z$ is differentiable almost everywhere and $\|\dot z (\cdot)\|\leq L\|\dot x(\cdot)\|$ holds almost everywhere.
\end{remark}

The following definition specifies which type of solutions we consider in the analysis of the dynamical system \eqref{ADMMdysy-subdiff}.

\begin{definition}\label{str-sol} Let $(x^0,z^0,y^0) \in {\mathcal H}\times{\mathcal G}\times{\mathcal G}$, $c >0$, $\gamma \in [0,1]$, and $M_1 : [0,+\infty) \rightarrow {\cal S}_+({\cal H})$ and $M_2 : [0,+\infty) \rightarrow {\cal S}_+({\cal G})$. We say that the function $(x,z,y):[0,+\infty)\To {\mathcal H}\times{\mathcal G}\times{\mathcal G}$ is a strong global solutions of \eqref{ADMMdysy-subdiff}, if the
following properties are satisfied:
\begin{enumerate}
\item[(i)] the functions $x,z,y$ are locally absolutely continuous;

\item[(ii)] for almost every $t \in [0, +\infty)$
\begin{align*}
\dot{x}(t)+x(t) & \in\left(\partial f +cA^*A+M_1(t)\right)^{-1}\left(M_1(t)x(t)+cA^*z(t)-A^*y(t)-\n h(x(t))\right),\\ 
\dot{z}(t)+z(t) & \in \left(\partial g +c\id+M_2(t)\right)^{-1}\left(M_2(t)z(t)+cA(\gamma\dot x(t)+x(t))+y(t)\right)\\
\dot{y}(t) & =c A(x(t)+\dot{x}(t))-c (z(t)+\dot{z}(t));
\end{align*}

\item[(iii)] $x(0)=x^0,\,z(0)=z^0,\mbox{ and }y(0)=y^0.$
\end{enumerate}
\end{definition}

The following results  will be useful in the proof of  the existence and uniqueness theorem.

\begin{lemma}\label{lipschitz1} Assume that $(Cweak)$ holds. Then, for every fixed $t\in[0,+\infty)$, the operator
$$S_{t}:\mathcal{H}\To\mathcal{H}, \quad S{_t}(u)=\argmin_{x\in\mathcal{H}}\left(F(t,x)+\frac{c}{2}\|x-u\|^2\right),$$
is Lipschitz continuous.
\end{lemma}
\begin{proof} Let $t\in[0,+\infty)$ be fixed and $u,v\in\mathcal{H}$. By subdifferential calculus we obtain that
$$cu\in \p f(S{_t}(u))+ \big (cA^*A + M_1(t) \big) (S{_t}(u))$$
and
$$cv\in \p f(S{_t}(v))+\big (cA^*A + M_1(t) \big) (S{_t}(v)).$$
Using that, due to $(Cweak)$, $\partial f  + cA^*A + M_1(t)$ is $\alpha(t)$-strongly monotone, we get
$$\a(t)\|S{_t}u-S{_t}v\|^2 \leq c \left \langle u-v, S_t(u) - S_t(v) \right \rangle.$$
By the Cauchy-Schwartz inequality we obtain
$$\|S{_t}u-S{_t}v\|\le \frac{c}{\alpha(t)}\|u-v\|,$$
which shows that $S_t$ is Lipschitz continuous with constant $\frac{c}{\alpha(t)}$.
\end{proof}

Now we are going to prove another technical result which will be used in the proof of the main theorem of this section.

\begin{lemma}\label{lipschitz2}  Assume that $(Cweak)$ holds. Let be $(x,z,y)\in\mathcal{H}\times\mathcal{G}\times\mathcal{G}$ and the maps $R_{(x,z,y)}:[0,+\infty)\To\mathcal{H},$ 
$$R_{(x,z,y)}(t)=\argmin_{u\in\mathcal{H}}\left(F(t,u)+\frac{c}{2}\left\|u-\left(\frac{1}{c}M_1(t)x+A^*z-\frac{1}{c}A^*y-\frac{1}{c}\n h(x)\right)\right\|^2\right)-x,$$
and  $Q_{(x,z,y)}:[0,+\infty)\To\mathcal{G},$ 
$$Q_{(x,z,y)}(t)=\argmin_{v\in\mathcal{G}}\left(G(t,v)+\frac{c}{2}\left\|v-\left(\frac{1}{c}M_2(t)z+A\big(\g R_{(x,z,y)}(t)+x\big)+\frac{1}{c}y\right)\right\|^2\right)-z.$$
Then the following statements are true for every $t,r \in [0,+\infty)$:
\begin{itemize}
\item[(i)]$\|R_{(x,z,y)}(t)-R_{(x,z,y)}(r)\|\le \frac{\|R_{(x,z,y)}(r)\|}{\a(t)} \|M_1(t)-M_1(r)\|;$

\item[(ii)] $\|Q_{(x,z,y)}(t)-Q_{(x,z,y)}(r)\|\le \frac{\|Q_{(x,z,y)}(r)\|}{c}\|M_2(t)-M_2(r)\|+\frac{\g\|A\| \|R_{(x,z,y)}(r)\|}{\a(t)} \|M_1(t)-M_1(r)\|.$
\end{itemize}
\end{lemma}
\begin{proof} Let $t,r\in[0,+\infty)$ be fixed.

(i) From the definition of $R_{(x,z,y)}$ one has
$$M_1(t)x+cA^*z-A^*y-\n h(x)\in \p f(R_{(x,z,y)}(t)+x)+ \big(cA^*A + M_1(t)\big )(R_{(x,z,y)}(t)+x)$$
and
$$M_1(r)x+cA^*z-A^*y-\n h(x)\in \p f(R_{(x,z,y)}(r)+x)+ \big(cA^*A + M_1(r)\big )(R_{(x,z,y)}(r)+x),$$
which is equivalent to
\begin{align*}
M_1(t)(R_{(x,z,y)}(r)+x)-M_1(r)(R_{(x,z,y)}(r))+cA^*z-A^*y-\n h(x) & \in\\
\p f(R_{(x,z,y)}(r)+x)+\big(cA^*A + M_1(t))(R_{(x,z,y)}(r)+x). &
\end{align*}
Using again that $\p f+cA^*A+M_1(t)$ is $\alpha(t)$-strongly monotone for every $t \in [0,+\infty)$, we obtain
$$\<M_1(t)(R_{(x,z,y)}(r))-M_1(r)(R_{(x,z,y)}(r)),R_{(x,z,y)}(r)-R_{(x,z,y)}(t)\>\ge \a(t)\|R_{(x,z,y)}(r)-R_{(x,z,y)}(t)\|^2.$$
The conclusion follows via the Cauchy-Schwarz inequality.

(ii) From the definition of $Q_{(x,z,y)}$ one has
$$M_2(t)z+cA(\g R_{(x,z,y)}(t)+x)+y\in \p g(Q_{(x,z,y)}(t)+z)+ \big(M_2(t) + cI\big)(Q_{(x,z,y)}(t)+z)$$
and
$$M_2(r)z+cA(\g R_{(x,z,y)}(r)+x)+y\in \p g(Q_{(x,z,y)}(r)+z)+ \big(M_2(r) + cI\big)(Q_{(x,z,y)}(r)+z),$$
which is equivalent to
\begin{align*}
-M_2(r)(Q_{(x,z,y)}(r))+M_2(t)(Q_{(x,z,y)}(r)+z)+cA(\g R_{(x,z,y)}(r)+x)+y &\in\\
\p g(Q_{(x,z,y)}(r)+z)+ \big(M_2(t) + cI \big)(Q_{(x,z,y)}(r)+z). &
\end{align*}
Using that $\p g+M_2(t)+cI$ is $c-$strongly monotone, we obtain
\begin{align*}
\<M_2(t)(Q_{(x,z,y)}(r))-M_2(r)(Q_{(x,z,y)}(r))+c\g A(R_{(x,z,y)}(r)-R_{(x,z,y)}(t)),Q_{(x,z,y)}(r)-Q_{(x,z,y)}(t)\> & \ge\\
c\|Q_{(x,z,y)}(r)-Q_{(x,z,y)}(t)\|^2.&
\end{align*}
From the Cauchy-Schwarz inequality and (i) it follows
$$\|Q_{(x,z,y)}(r)-Q_{(x,z,y)}(t)\|\le \frac{\|Q_{(x,z,y)}(r)\|}{c}\|M_2(t)-M_2(r)\|+\frac{\g\|A\| \|R_{(x,z,y)}(r)\|}{\a(t)}\|M_1(t)-M_1(r)\|.$$
\end{proof}

Now we can prove existence and uniqueness of a strong global solution of \eqref{ADMMdysy-subdiff} under $(CStrong)$. 
To this end we will first reformulate  \eqref{ADMMdysy} as a particular first order dynamical system in a suitably chosen product space (see also \cite{alv-att-bolte-red}). Subsequently
we will make use of the Cauchy-Lipschitz-Picard Theorem for absolutely continues trajectories (see, for example, \cite[Proposition 6.2.1]{haraux}, \cite[Theorem 54]{sontag}).  Notice that under $(CStrong)$  the operator $S_t$ in Lemma \ref{lipschitz1} is Lipschitz continuous with constant $\frac{c}{\alpha}$ for every $t\in[0,+\infty)$.

\begin{theorem}\label{uniq} Assume that $(Cstrong)$ holds, and $M_1 \in L^1_{loc}([0,+\infty), {\cal H})$ and $M_2 \in L^1_{loc}([0,+\infty), {\cal G})$, namely,
$$t\To\|M_1(t)\| \ \mbox{and} \ t\To\|M_2(t)\|$$ are integrable on $[0,T]$ for every $T>0$. 
Then, for every starting points $(x^0,z^0,y^0)\in {\mathcal H}\times{\mathcal G}\times{\mathcal G}$, the dynamical system \eqref{ADMMdysy-subdiff} has a unique strong global solution  $(x,z,y):[0,+\infty)\To {\mathcal H}\times{\mathcal G}\times{\mathcal G}.$
\end{theorem}
\begin{proof}
Denoting $U(t)=(x(t),z(t),y(t))$, the dynamical system \eqref{ADMMdysy-subdiff} can be rewritten as
\begin{equation}\label{dysy1}
\left\{
\begin{array}{ll}
\dot{U}(t)=\G(t,U(t))\\
U(0)=(x^0,z^0,y^0),
\end{array}
\right.
\end{equation}
where 
$$\G:[0,+\infty)\times{\mathcal H}\times{\mathcal G}\times{\mathcal G}\To {\mathcal H}\times{\mathcal G}\times{\mathcal G}, \quad \G(t,x,z,y)=\left(u,v,w\right),$$
is defined as
\begin{align*}
u=u(t,x,z,y) & =\argmin_{a\in\mathcal{H}}\left(F(t,a)+\frac{c}{2}\left\|a-\left(\frac{1}{c}M_1(t)x+A^*z-\frac{1}{c}A^*y-\frac{1}{c}\n h(x)\right)\right\|^2\right)-x\\
v=v(t,x,z,y) & =\argmin_{b\in\mathcal{G}}\left(G(t,b)+\frac{c}{2}\left\|b-\left(\frac{1}{c}M_2(t)z+A(\g u+x)+\frac{1}{c}y\right)\right\|^2\right)-z\\
& = \prox\nolimits_{\frac{1}{c}G(t,\cdot)}\left(\frac{1}{c}M_2(t)z+A(\g u+x)+\frac{1}{c}y\right)-z,\\
w=w(t,x,z,y) & = c A(u+x)-c (v+z).
\end{align*}
The existence and uniqueness of a strong global solution follows according to the Cauchy-Lipschitz-Picard Theorem, if we show: (1) that $\G(t,\cdot,\cdot,\cdot)$ is $L(t)$-Lipschitz continuous for every $t\in [0,+\infty)$ 
and the Lipschitz constant as a function of time has the property that $L(\cdot)\in L^1_{loc}([0,+\infty), \R)$; (2) that $\G(\cdot,x,z,y)\in L^1_{loc}([0,+\infty),\mathcal{H}\times \mathcal{G}\times \mathcal{G})$ for every $(x,z,y)\in \mathcal{H}\times \mathcal{G}\times \mathcal{G}$.

(1) Let $t\in [0,+\infty)$ be fixed and consider $(x,z,y),(\ol x,\ol z,\ol y)\in \mathcal{H}\times \mathcal{G}\times \mathcal{G}$. We have
$$\|\G(t,x,z,y)-\G(t,\ol x,\ol z,\ol y)\|=\sqrt{\|u-\ol u\|^2+\|v-\ol v\|^2+\|w-\ol w\|^2},$$
where (see Lemma \ref{lipschitz1})
\begin{align*}
u-\ol u = & \ \argmin_{a\in\mathcal{H}}\left(F(t,a)+\frac{c}{2}\left\|a-\left(\frac{1}{c}M_1(t)x+A^*z-\frac{1}{c}A^*y-\frac{1}{c}\n h(x)\right)\right\|^2\right)\\
 & - \argmin_{a\in\mathcal{H}}\left(F(t,a)+\frac{c}{2}\left\|a-\left(\frac{1}{c}M_1(t)\ol x+A^*\ol z-\frac{1}{c}A^*\ol y-\frac{1}{c}\n h(\ol x)\right)\right\|^2\right)+\ol x-x\\
= & \ S{_t}\left(\frac{1}{c}M_1(t)x+A^*z-\frac{1}{c}A^*y-\frac{1}{c}\n h(x)\right)-S{_t}\left(\frac{1}{c}M_1(t)\ol x+A^*\ol z-\frac{1}{c}A^*\ol y-\frac{1}{c}\n h(\ol x)\right)\\
& +\ol x-x.
\end{align*}
Hence,
\begin{align*}
\|u-\ol u\|^2 \le & \ 2\left\|S{_t}\left(\frac{M_1(t)}{c}x+A^*z-\frac{A^*}{c}y-\frac{1}{c}\n h(x)\right)-S{_t}\left(\frac{M_1(t)}{c}\ol x+A^*\ol z-\frac{A^*}{c}\ol y-\frac{1}{c}\n h(\ol x)\right)\right\|^2\\
& + 2\|\ol x-x\|^2.
\end{align*}
Using Lemma \ref{lipschitz1} and taking into account that $(Cstrong)$ is fulfilled, which means that the Lipschitz constant of the operator $S_t$ is $\frac{c}{\a}$, it follows
\begin{align*} 
 \|u-\ol u\|^2 \le & \ 2\frac{c^2}{\a^2}\left\|\frac{1}{c}M_1(t)(x-\ol x)+A^*(z-\ol z)-\frac{1}{c}A^*(y-\ol y)-\frac{1}{c}(\n h(x)-\n h(\ol x))\right\|^2+2\|\ol x-x\|^2\\
 \le & \ 2 \frac{c^2}{\a^2} \left(4\frac{\|M_1(t)\|^2}{c^2}\|x- \ol x\|^2+4\|A\|^2\|z-\ol z\|^2+4\frac{\|A\|^2}{c^2}\|y-\ol y\|^2+\frac{4}{c^2}\|\n h(x)-\n h(\ol x)\|^2\right)\\
 & + 2\|x- \ol x\|^2 \\
   \le & \ 2\left(\frac{4\|M_1(t)\|^2+4L_h^2}{\a^2}+1\right)\|x- \ol x\|^2+\frac{8c^2}{\a^2} \|A\|^2\|z-\ol z\|^2+\frac{8}{\a^2}\|A\|^2\|y-\ol y\|^2.
 \end{align*}
By taking into account the nonexpansiveness of the proximal operator and that $\gamma \in [0,1]$, it also follows
\begin{align*}
\|v-\ol v\|^2 \leq & \ 2 \left\|\prox\nolimits_{\frac{1}{c}G(t,\cdot)}\!\!\left(\!\frac{1}{c}M_2(t)z+A(\g u+x)+\frac{1}{c}y\right)\!-\!\prox\nolimits_{\!\frac{1}{c}G(t,\cdot)}\!\!\left(\frac{1}{c}M_2(t)\ol z+A(\g\ol u+\ol x)+\frac{1}{c}\ol y\right)\!\right\|^2\\
& + 2\|z-\ol z\|^2 \\
\leq & \ 2\left\|\frac{1}{c}M_2(t)(z-\ol z)+A(\g(u-\ol u)+x-\ol x)+\frac{1}{c}(y-\ol y)\right\|^2+2\|z-\ol z\|^2 \\
\leq & \ \frac{8\|M_2(t)\|^2}{c^2}\|z-\ol z\|^2+8\g^2\|A\|^2\|u-\ol u\|^2+8\|A\|^2\|x-\ol x\|^2+\frac{8}{c^2}\|y-\ol y\|^2+2\|z-\ol z\|^2 \\
\leq & \ 8 \|A\|^2 \left(\frac{8\|M_1(t)\|^2+8L_h^2}{\a^2}+3\right)\|x- \ol x\|^2 + \left(\frac{64c^2}{\alpha^2} \|A\|^4 + \frac{8\|M_2(t)\|^2}{c^2}+2\right)\|z-\ol z\|^2\\
& + \left(\frac{64}{\alpha^2} \|A\|^4 + \frac{8}{c^2}\right)\|y-\ol y\|^2.
\end{align*}
Finally,
\begin{align*}
\|w-\ol w\|^2= & \ \|c A(u-\ol u+x-\ol x)-c (v-\ol v+z-\ol z)\|^2\\
\leq & \ 4c^2\|A\|^2\|u-\ol u\|^2+4c^2\|A\|^2\|x-\ol x\|^2+4c^2\|v-\ol v\|^2+4c^2\|z-\ol z\|^2\\
\leq & \ 36c^2 \|A\|^2 \left(\frac{8\|M_1(t)\|^2+8L_h^2}{\a^2} + 3\right) \|x- \ol x\|^2\\
& + 4 \left(\frac{72c^4}{\alpha^2} \|A\|^4 + 8 \|M_2(t)\|^2 + 3c^2\right)\|z-\ol z\|^2 + 32 \left(\frac{9c^2}{\alpha^2} \|A\|^4 + 1\right)\|y-\ol y\|^2.
\end{align*}
Consequently,
\begin{align*}
\|\G(t,x,z,y)-\G(t,\ol x,\ol z,\ol y)\| \le & \  \sqrt{L_1(t)\|x-\ol x\|^2+L_2(t)\|z-\ol z\|^2+L_3(t)\|y-\ol y\|^2}\|\\
\le & \ \sqrt{L_1(t)+L_2(t)+L_3(t)}\sqrt{\|x-\ol x\|^2+\|z-\ol z\|^2+\|y-\ol y\|^2} \\
= & \ L(t) \|(x,z,y)- (\ol x, \ol z, \ol y)\|,
\end{align*}
where
$$L(t) =  \sqrt{L_1(t)+L_2(t)+L_3(t)}$$
and 
\begin{align*}
L_1 (t) = & \ 2\left(\frac{4\|M_1(t)\|^2+4L_h^2}{\a^2}+1\right) + 8 \|A\|^2 \left(\frac{8\|M_1(t)\|^2+8L_h^2}{\a^2}+3\right) \\
& + 36c^2 \|A\|^2 \left(\frac{8\|M_1(t)\|^2+8L_h^2}{\a^2} + 3\right),\\
L_2(t) = & \ \frac{8c^2}{\a^2} \|A\|^2 + \left(\frac{64c^2}{\alpha^2} \|A\|^4 + \frac{8\|M_2(t)\|^2}{c^2}+2\right) + 4 \left(\frac{72c^4}{\alpha^2} \|A\|^4 + 8 \|M_2(t)\|^2 + 3c^2\right),\\
L_3(t) = & \ \frac{8}{\a^2}\|A\|^2 + \left(\frac{64}{\alpha^2} \|A\|^4 + \frac{8}{c^2}\right) + 32 \left(\frac{9c^2}{\alpha^2} \|A\|^4 + 1\right),
\end{align*}
which means that $\G(t,\cdot,\cdot,\cdot)$ is $L(t)$-Lipschitz continuous. Since $M_1 \in L^1_{loc}([0,+\infty), {\cal H})$ and $M_2 \in L^1_{loc}([0,+\infty), {\cal G})$, it is obvious that $L(\cdot)\in L^1_{loc}([0,+\infty), \R)$.

(2) Now we will show that  $\G(\cdot,x,z,y)\in L^1_{loc}([0,+\infty),\mathcal{H}\times \mathcal{G}\times \mathcal{G})$ for every $(x,z,y)\in \mathcal{H}\times \mathcal{G}\times \mathcal{G}$. Let $(x,z,y)\in\mathcal{H}\times \mathcal{G}\times \mathcal{G}$ be fixed and $T>0$. We have
$$\int_0^T\|\G(t,x,z,y)\|dt=\int_0^T\sqrt{\|u(t,x,z,y)\|^2+\|v(t,x,z,y)\|^2+\|w(t,x,z,y)\|^2}dt.$$
By Lemma \ref{lipschitz2} and taking into account that $\alpha(t) = \alpha >0$ for every $t \in [0,+\infty)$ and  $\gamma \in [0,1]$, we have for every $t \in [0,+\infty)$ that
\begin{align*}
\|u(t,x,z,y)\|^2\le & \ 2\|u(t,x,z,y)-u(0,x,z,y)\|^2+2\|u(0,x,z,y)\|^2\\
\le & \ \frac{2\|u(0,x,z,y)\|^2}{\a^2} \|M_1(t)-M_1(0)\|^2+2\|u(0,x,z,y)\|^2,
\end{align*}
\begin{align*}
\|v(t,x,z,y)\|^2\le & \ 2\|v(t,x,z,y)-v(0,x,z,y)\|^2+2\|v(0,x,z,y)\|^2\\
\le & \ \frac{4 \|v(0,x,z,y)\|^2}{c^2} \|M_2(t)-M_2(0)\|^2+  \frac{4 \|A\|^2\|u(0,x,z,y)\|^2}{\a^2} \|M_1(t)-M_1(0)\|^2\\
& +  2\|v(0,x,z,y)\|^2
\end{align*}
and
\begin{align*}
\|w(t,x,z,y)\|^2= & \ c^2\|(Au(t,x,z,y)+x)-(v(t,x,z,y)+z)\|^2\\
 \le & \  3c^2(\|A\|^2\|u(t,x,z,y)\|^2+\|v(t,x,z,y)\|^2+\|Ax-z\|^2)\\
\le & \ \frac{18 c^2 \|A\|^2\|u(0,x,z,y)\|^2}{\a^2}  \|M_1(t)-M_1(0)\|^2 +12 \|v(0,x,z,y)\|^2 \|M_2(t)-M_2(0)\|^2\\
& \ + 3c^2 \big(2 \|A\|^2 \|u(0,x,z,y)\|^2 + 2\|v(0,x,z,y)\|^2 + \|x-z\|^2 \big).
\end{align*}
Since  $M_1 \in L^1_{loc}([0,+\infty), {\cal H})$ and $M_2 \in L^1_{loc}([0,+\infty), {\cal G})$, it follows that the integral
$$\int_0^T\|\G(t,x,z,y)\|dt$$ 
exists and it is finite, in other words,
$\G(\cdot,x,z,y)\in L^1_{loc}([0,+\infty),\mathcal{H}\times \mathcal{G}\times \mathcal{G})$. 

Consequently, the dynamical system \eqref{dysy1} has a unique locally absolutely continuous solution, which means that the dynamical system \eqref{ADMMdysy} has a unique strong global solution.
\end{proof}

\section{Some technical results}\label{sectech}

In this section we will prove some technical results which will be useful in the asymptotic analysis of the dynamical system \eqref{ADMMdysy-subdiff}. We endow the real linear space $\mathcal{L}(\mathcal{H}):=\{A:\mathcal{H}\To\mathcal{H}:A\mbox{ is linear and continuous}\}$ with the norm
 $$\|A\|=\sup_{\|x\|\le 1}\|Ax\|.$$
If $A\in\mathcal{L}(\mathcal H)$ is self-adjoint, then it holds (see \cite[Lemma 3.2.4 iv)]{zimmer})
\begin{equation*}
\|A\|=\sup_{\|x\|\le 1}|\<Ax,x\>|.
\end{equation*}
\begin{definition}\label{uniformderiv} We say that the map $M:[0,+\infty)\To \mathcal{L}(\mathcal{H}),\,t\To M(t),$ is  derivable at $t_0\in[0,+\infty)$, if the limit
$$\lim_{h\To 0}\frac{M(t_0+h)-M(t_0)}{h}$$
taken with respect to the norm topology of  $\mathcal{L}(\mathcal{H})$ exists. In this case we denote by $\dot{M}(t_0) \in \mathcal{L}(\mathcal{H})$ the value of this limit.
\end{definition}
If $\dot{M}(t_0)$ exists, for $t_0\in[0,+\infty)$, then one can easily see that
$$\dot{M}(t_0)x=\lim_{h\To 0}\frac{M(t_0+h)x-M(t_0)x}{h} \ \mbox{ for every }x\in\mathcal{H}.$$
According to Remark \ref{rem-abs-cont}, if $M$ is locally absolutely continuous then
$\dot{M}(t)$  exists for almost every $t\in[0,+\infty).$

Assume now that $M(t)\in\mathcal{L}(\mathcal H)$ is self-adjoint for every $t\in[0,+\infty)$ and that it is derivable at $t_0\in[0,+\infty)$. For all $x,u \in {\cal H}$ we have
\begin{align*}
\<\dot{M}(t_0)x,u\> & =\left\<\lim_{h\To 0}\frac{M(t_0+h)x-M(t_0)x}{h},u\right\>=\lim_{h\To 0}\left\<\frac{M(t_0+h)x-M(t_0)x}{h},u\right\>\\
& = \lim_{h\To 0}\left\<x,\frac{M(t_0+h)u-M(t_0)u}{h}\right\>=\<x,\dot{M}(t_0)u\>,
\end{align*}
which shows that $\dot{M}(t_0)$ is also self-adjoint.

\begin{lemma}\label{compderiv} Let $M:[0,+\infty)\To \mathcal{L}(\mathcal{H}),\,t\To M(t),$ be derivable at $t_0\in[0,+\infty)$,  and  let the maps $x,y:[0,+\infty)\To\mathcal{H}$ be also derivable at $t_0.$ Then the real function
$t\To \<M(t)x(t),y(t)\>$ is derivable at $t_0$ and one has
\begin{equation*}
\frac{d}{dt}\<M(t)x(t),y(t)\>\big|_{t=t_0}=\<\dot{M}(t_0)x(t_0),y(t_0)\>+\<M(t_0)\dot{x}(t_0),y(t_0)\>+\<M(t_0){x}(t_0),\dot{y}(t_0)\>.
\end{equation*}
\end{lemma}
\begin{proof} We have
\begin{align*}
\frac{d}{dt}M(t)x(t)\big|_{t=t_0} & = \lim_{h\To 0}\frac{M(t_0+h)x(t_0+h)-M(t_0)x(t_0)}{h}\\
& = \lim_{h\To 0}M(t_0+h)\left(\frac{x(t_0+h)-x(t_0)}{h}\right)+\lim_{h\To 0}\frac{M(t_0+h)x(t_0)-M(t_0)x(t_0)}{h}\\
& = M(t_0)\dot{x}(t_0)+\dot{M}(t_0)x(t_0).
\end{align*}
The derivation formula of the scalar product leads to the desired conclusion
\begin{align*}
\frac{d}{dt}\<M(t)x(t),y(t)\>\big|_{t=t_0} & =\left<\frac{d}{dt}M(t)x(t)\big|_{t=t_0},y(t_0)\right\>+\<M(t_0)x(t_0),\dot{y}(t_0)\>\\
& = \<M(t_0)\dot{x}(t_0),y(t_0)\>+\<\dot{M}(t_0)x(t_0),y(t_0)\>+\<M(t_0){x}(t_0),\dot{y}(t_0)\>.
\end{align*}
\end{proof}

The main result of this section follows.

\begin{lemma}\label{abscont}  Assume that $(Cstrong)$ holds and that the maps $M_1:[0,+\infty)\To S_+(\mathcal{H})$ and $M_2:[0,+\infty)\To S_+(\mathcal{G})$ are  locally absolutely continuous. For a given starting point $(x^0,z^0,y^0)\in {\mathcal H}\times{\mathcal G}\times{\mathcal G}$, let  $(x,z,y):[0,+\infty)\To {\mathcal H}\times{\mathcal G}\times{\mathcal G}$ be the unique strong global solution of the dynamical system \eqref{ADMMdysy-subdiff}. Then 
$$t\To(\dot{x}(t),\dot{z}(t),\dot{y}(t))$$ is locally absolutely continuous, hence $(\ddot{x}(t),\ddot{z}(t),\ddot{y}(t))$ exists for almost every $t\in [0,+\infty)$. 

In addition, if $\sup_{t\ge 0}\|M_1(t)\|<+\infty$ and $\sup_{t\ge 0}\|M_2(t)\|<+\infty$, then there exists $L>0$ such that
$$
\|\ddot{x}(t)\|+\|\ddot{z}(t)\|+\|\ddot{y}(t)\|\le L\big(\|\dot{x}(t)\|+\|\dot{z}(t)\|+\|\dot{y}(t)\|+\|\dot{M_1}(t)\|\|\dot{x}(t)\|+\|\dot{M_2}(t)\|\|\dot{z}(t)\|\big)
$$
 for almost  every $t\in[0,+\infty).$
\end{lemma}

\begin{proof}
Let $T>0$ be fixed. We will use the same notations as in the proof of Theorem \ref{uniq}. Let $t,r\in [0,T]$ be fixed. We have
\begin{align*}
& \ \|\dot{U}(t)-\dot{U}(r)\|  =  \|\G(t,U(t))-\G(r,U(r))\| \le \!\|\G(t,U(t))-\G(t,U(r))\| +\|\G(t,U(r))-\G(r,U(r))\|\\
\leq & \ \|u(t,x(t),z(t),y(t)) - u(t, x(r), z(r), y(r)) \|+ \|v(t,x(t),z(t),y(t)) - v(t, x(r), z(r), y(r)) \| \\
& + \|w(t,x(t),z(t),y(t)) - w(t, x(r), z(r), y(r)) \| \\
& + \|u(t,x(r),z(r),y(r)) - u(r, x(r), z(r), y(r)) \| + \|v(t,x(r),z(r),y(r)) - v(r, x(r), z(r), y(r)) \| \\
& +\|w(t,x(r),z(r),y(r)) - w(r, x(r), z(r), y(r)) \|.
\end{align*}
Since
\begin{align*}
u(t,x(t),z(t),y(t)) - u(t, x(r), z(r), y(r)) = & \ S{_t}\left(\frac{1}{c}M_1(t)x(t)+A^*z(t)-\frac{1}{c}A^*y(t)-\frac{1}{c}\n h(x(t))\right)\\
 & - S{_t}\left(\frac{1}{c}M_1(t)x(r)+A^*z(r)-\frac{1}{c}A^*y(r)-\frac{1}{c}\n h(x(r))\right)\\
& - x(t) + x(r),
\end{align*}
according to Lemma 1, we get
\begin{align*}
& \ \|u(t,x(t),z(t),y(t)) - u(t, x(r), z(r), y(r)) \|  \\
\le & \left(\frac{\|M_1(t)\|}{\alpha}+\frac{L_h}{\alpha}+1\right)\|x(t)-x(r)\|+\frac{c}{\a}\|A\|\|z(t)-z(r)\|+\frac{\|A\|}{\alpha}\|y(t)-y(r)\|.
\end{align*}
Since $t\To\|M_1(t)\|$ is bounded on $[0,T],$ there exists $L_1:=L_1(T) > 0$ such that
\begin{equation}\label{useu}
 \|u(t,x(t),z(t),y(t)) - u(t, x(r), z(r), y(r)) \| \le L_1(\|x(t)-x(r)\|+\|z(t)-z(r)\|+\|y(t)-y(r)\|).
\end{equation}
Similarly, since
\begin{align*}
& \ v(t,x(t),z(t),y(t)) - v(t, x(r), z(r), y(r))\\
= & \ \prox\nolimits_{\frac{1}{c}G(t,\cdot)}\left(\frac{1}{c}M_2(t)z(t)+A(\g u (t, x(t), z(t), y(t)) +x(t))+\frac{1}{c}y(t)\right)\\
& - \prox\nolimits_{\frac{1}{c}G(t,\cdot)}\left(\frac{1}{c}M_2(t)z(r)+A(\g u (t, x(r), z(r), y(r))+x(r))+\frac{1}{c}y(r)\right) - z(t) + z(r),
\end{align*}
by the nonexpansiveness of the proximal operator we get
\begin{align*}
& \ \|v(t,x(t),z(t),y(t)) - v(t, x(r), z(r), y(r)) \|  \\
\le & \left(\frac{\|M_2(t)\|}{c}+1\right)\|z(t)-z(r)\|+\|A\|\|x(t)-x(r)\|+\frac{1}{c}\|y(t)-y(r)\|\\
& +\g\|A\|\|u(t,x(t),z(t),y(t)) - u(t, x(r), z(r), y(r)) \|.
\end{align*}
Since $t\To\|M_2(t)\|$ is bounded on $[0,T],$  by taking into consideration \eqref{useu},
one can easily see that there exists $L_2:=L_2(T) > 0$ such that
\begin{equation}\label{usev}
\|v(t,x(t),z(t),y(t)) - v(t, x(r), z(r), y(r)) \| \le L_2(\|x(t)-x(r)\|+\|z(t)-z(r)\|+\|y(t)-y(r)\|).
\end{equation}
Further, by using \eqref{useu} and \eqref{usev}, we get
\begin{align*}
& \ \|w(t,x(t),z(t),y(t)) - w(t, x(r), z(r), y(r)) \|  \\
\leq & \ c \| A(u(t,x(t), z(t), y(t))-u(t, x(r), z(r), y(r)) + x(t)-x(r))\| \\ 
& + c \|v(t,x(t), z(t), y(t))-v(t, x(r), z(r), y(r)) +z(t)-z(r)\|\\
\leq & \ c(\|A\|L_1+\|A\|+L_2)\|x(t)-x(r)\|+c (\|A\|L_1+L_2+1)\|z(t)-z(r)\|\\
& + c(\|A\|L_1+L_2)\|y(t)-y(r)\|.
\end{align*}
Hence, there exists $L_3:=  c(\|A\|L_1+\|A\|+L_2 + 1) > 0$ such that
\begin{equation}\label{usew}
\|w(t,x(t),z(t),y(t)) - w(t, x(r), z(r), y(r)) \| \le L_3(\|x(t)-x(r)\|+\|z(t)-z(r)\|+\|y(t)-y(r)\|).
\end{equation}
Using now Lemma \ref{lipschitz2} (i), we get
\begin{equation}\label{useu11}
\begin{array}{rl}
\|u(t,x(r),z(r),y(r)) - u(r, x(r), z(r), y(r))\| = & \|R_{(x(r),z(r),y(r))}(t)-R_{(x(r),z(r),y(r))}(r)\|\\
\leq & \frac{\|R_{(x(r),z(r),y(r))}(r)\|}{\a} \|M_1(t)-M_1(r)\|.
\end{array}
\end{equation}
Since  $r \mapsto S_r$ and $\n h$ are Lipschitz continuous and $x,z,y$ and $M_1$ are absolutely continuous,  the map
$$r \mapsto R_{(x(r),z(r),y(r))}(r)  = S{_r}\left(\frac{1}{c}M_1(r)x(r)+A^*z(r)-\frac{1}{c}A^*y(r)-\frac{1}{c}\n h(x(r))\right)-x(r)$$
is bounded on $[0,T].$ Consequently, there exists $L_4:=L_4(T)>0$ such that
\begin{equation}\label{useu1}
\|u(t,x(r),z(r),y(r)) - u(r, x(r), z(r), y(r))\| \le L_4\|M_1(t)-M_1(r)\|.
\end{equation}
Similarly, using this time Lemma \ref{lipschitz2} (ii), we get
\begin{equation}\label{usev11}
\begin{array}{rl}
& \ \|v(t,x(r),z(r),y(r)) - v(r, x(r), z(r), y(r))\| =  \|Q_{(x(r),z(r),y(r))}(t)-Q_{(x(r),z(r),y(r))}(r)\|\\
\leq & \ \frac{\|A\|\|R_{(x(r),z(r),y(r))}(r)\|}{\a} \|M_1(t)-M_1(r)\| +  \frac{\|Q_{(x(r),z(r),y(r))}(r)\|}{c} \|M_2(t)-M_2(r)\|.
\end{array}
\end{equation}
Since the proximal operator is nonexpansive and $x,z,y$ and $M_2$ are absolutely continuous, the map
$$r \mapsto Q_{(x(r),z(r),y(r))}(r)  = \prox\nolimits_{\frac{1}{c}G(r,\cdot)}\left(\frac{1}{c}M_2(r)z(r)+A(\g u (r, x(r), z(r), y(r)) +x(r))+\frac{1}{c}y(r)\right)-z(r)$$
is bounded on $[0,T].$ Consequently, there exists $L_5:=L_5(T) >0$ such that
\begin{equation}\label{usev1}
\|v(t,x(r),z(r),y(r)) - v(r, x(r), z(r), y(r))\| \le L_5(\|M_1(t)-M_1(r)\|+\|M_2(t)-M_2(r)\|).
\end{equation}
Further, by using \eqref{useu1} and \eqref{usev1}, we get
\begin{align*}
& \ \|w(t,x(r),z(r),y(r)) - w(r, x(r), z(r), y(r)) \|  \\
\leq & \ c \| A(u(t,x(r), z(r), y(r))-u(r, x(r), z(r), y(r)) )\| + c \|v(t,x(r), z(r), y(r))-v(r, x(r), z(r), y(r))\|\\
\leq & \ c(\|A\|L_4+L_5)\|M_1(t)-M_1(r)\|+cL_5\|M_2(t)-M_2(r)\|
\end{align*}
Consequently, there exists $L_6:= c(\|A\|L_4+L_5) >0$ such that
\begin{equation}\label{usew1}
\|w(t,x(r),z(r),y(r)) - w(r, x(r), z(r), y(r)) \|  \le L_6(\|M_1(t)-M_1(r)\| + \|M_2(t)-M_2(r)\|).
\end{equation}

Summing the relations \eqref{useu}-\eqref{usew1} we obtain that there exists $L_7 >0$ such that
\begin{align*}
& \ \|\dot{U}(t)-\dot{U}(r) \|  \\
\leq & \ L_7(\|x(t)-x(r)\|+\|z(t)-z(r)\|+\|y(t)-y(r)\|+\|M_1(t)-M_1(r)\|+\|M_2(t)-M_2(r)\|).
\end{align*}

Let be $\e>0$. Since the maps $x, z, y, M_1$ and $M_2$ are absolutely continuous on $[0,T]$, there exists $\eta >0$ such that for any finite family of intervals $I_k=(a_k,b_k)\subseteq [0,T]$ such that
for any subfamily of disjoint intervals $I_j$ with $\sum_j|b_j-a_j| < \eta$ it holds 
$$\sum_j\|{x}(b_j)-{x}(a_j)\| < \frac{\varepsilon}{5L_7},\,\sum_j\|{z}(b_j)-{z}(a_j)\| < \frac{\varepsilon}{5L_7},\,\sum_j\|{y}(b_j)-{y}(a_j)\| < \frac{\varepsilon}{5L_7},$$
$$
\sum_j\|{M_1}(b_j)-{M_1}(a_j)\|< \frac{\varepsilon}{5L_7}\mbox{ and }\sum_j\|{M_2}(b_j)-{M_2}(a_j)\|< \frac{\varepsilon}{5L_7}.$$
 Consequently,
$$\sum_j\|\dot{U}(b_j)-\dot{U}(a_j)\|  < \varepsilon,$$
hence $\dot{U}(\cdot)=(\dot{x}(\cdot),\dot{z}(\cdot),\dot{y}(\cdot))$ is absolutely continuous on $[0,T]$. This proves that the second order derivatives $\ddot{x},\ddot{z},\ddot{y}$ exist almost everywhere on $[0,+\infty).$

We come now to the proof of the second statement and assume to this end that $\sup_{t\ge 0}\|M_1(t)\|<+\infty$ and $\sup_{t\ge 0}\|M_2(t)\|<+\infty$. Under these assumption, $L_1, L_2$ and $L_3$ appearing in \eqref{useu}, \eqref{usev} and \eqref{usew}, respectively, can be taken as being global constants,  that is, \eqref{useu}-\eqref{usew} hold for every $t,r\in[0,+\infty)$.

Since $R_{(x(r),z(r),y(r))}(r) = \dot{x}(r)$ and $Q_{(x(r),z(r),y(r))}(r) = \dot{z}(r)$ for every $r \in [0, +\infty)$, from \eqref{useu11} and \eqref{usev11} we get
$$\|u(t,x(r),z(r),y(r)) - u(r, x(r), z(r), y(r))\|  \leq \frac{\|\dot{x}(r)\|}{\a} \|M_1(t)-M_1(r)\|$$
and, respectively, 
$$ \|v(t,x(r),z(r),y(r)) - v(r, x(r), z(r), y(r))\| \leq  \frac{\|A\|\|\dot{x}(r)\|}{\a} \|M_1(t)-M_1(r)\| +  \frac{\|\dot{z}(r)\|}{c} \|M_2(t)-M_2(r)\|$$
for every $t,r\in[0,+\infty).$
Consequently, 
\begin{align*}
& \ \|w(t,x(r),z(r),y(r)) - w(r, x(r), z(r), y(r)) \|  \\
\leq & \ c  \|A\| \| u(t,x(r), z(r), y(r))-u(r, x(r), z(r), y(r))\| + c \|v(t,x(r), z(r), y(r))-v(r, x(r), z(r), y(r))\|\\
\leq & \ \frac{2 c \|A\|}{\alpha} \|\dot x(r)\|\|M_1(t)-M_1(r)\|+ \|\dot z(r) \|\|M_2(t)-M_2(r)\|
\end{align*}
for every $t,r\in[0,+\infty).$ This shows that there exists $L>0$ such that
\begin{align*}
& \|\dot{U}(t)-\dot{U}(r)\| \leq \\
& \frac{L}{\sqrt{3}}\!(\|x(t)-x(r)\|\!+\!\|z(t)-z(r)\|\!+\!\|y(t)-y(r)\|\!+\!\|\dot{x}(r)\|\|M_1(t)-\!M_1(r)\|\!+\!\|\dot{z}(r)\|\|M_2(t)-\!M_2(r)\|)
\end{align*}
for every $t,r\in[0,+\infty).$

Now we fix $r \in [0,+\infty)$ at which the second derivative of the trajectories exist and take in the above inequality $t = r+h$ for some $h >0$. This yields
\begin{align*}
& \|\dot{x}(r+h)-\dot{x}(r)\|+\|\dot{z}(r+h)-\dot{z}(r)\|+\|\dot{y}(r+h)-\dot{y}(r)\|) \leq \sqrt{3} \|\dot{U}(r+h)-\dot{U}(r)\| \\
\leq & \ L(\|x(r+h)-x(r)\|+\|z(r+h)-z(r)\|+\|y(r+h)-y(r)\|)\\
 & +  L(\|\dot{x}(r)\|\|M_1(r+h)-M_1(r)\|+\|\dot{z}(r)\|\|M_2(r+h)-M_2(r)\|).
\end{align*}
 After dividing in the above inequality by $h$ and letting $h\To 0$, we obtain
 \begin{equation*}
\|\ddot{x}(r)\|+\|\ddot{z}(r)\|+\|\ddot{y}(r)\|\le L(\|\dot{x}(r)\|+\|\dot{z}(r)\|+\|\dot{y}(r)\|+\|\dot{x}(r)\|\|\dot{M_1}(r)\|+\|\dot{z}(r)\|\|\dot{M_2}(r)\|).
\end{equation*}
This inequality holds for almost every $r\in[0,+\infty).$
\end{proof}

\section{Asymptotic analysis}\label{sec3}

In this section we will address the asymptotic behaviour of the trajectories generated by the dynamical system \eqref{ADMMdysy-subdiff}. At the beginning we will recall two results which will play a central role in the asymptotic analysis (see \cite[Lemma 5.1]{abbas-att-sv} and \cite[Lemma 5.2]{abbas-att-sv}, respectively).

\begin{lemma}\label{fejer-cont1} Suppose that $A:[0,+\infty)\rightarrow\R$ is locally absolutely continuous and bounded from below and that
there exists $B\in L^1([0,+\infty), \R)$ such that for almost every $t \in [0,+\infty)$ $$\frac{d}{dt}A(t)\leq B(t).$$
Then there exists $\lim_{t\rightarrow +\infty} A(t)\in\R$.
\end{lemma}

\begin{lemma}\label{fejer-cont2}  If $1 \leq p < \infty$, $1 \leq r \leq \infty$, $A:[0,+\infty)\rightarrow[0,+\infty)$ is
locally absolutely continuous, $A\in L^p([0,+\infty), \R)$, $B:[0,+\infty)\rightarrow\R$, $B\in  L^r([0,+\infty), \R)$ and
for almost every $t \in [0,+\infty)$ $$\frac{d}{dt}A(t)\leq B(t),$$ then $\lim_{t\rightarrow +\infty} A(t)=0$.
\end{lemma}

The first result which we prove in this section is a continuous version of the Opial Lemma formulated in the setting of variable metrics (see \cite[Theorem 3.3]{CVu} for its discrete counterpart).

\begin{lemma}\label{opial-contvar} Let $\mathcal{C} \subseteq { {\mathcal H}}$ be a nonempty set and $x:[0,+\infty)\rightarrow{ {\mathcal H}}$ a continuous map. Let $M:[0,+\infty)\To S_+(\mathcal{H})$ be such that $M(t_1)\succcurlyeq M(t_2)$ for every $t_1,t_2\in [0,+\infty)$ with $t_1\le t_2$ and there exists  $\a>0$  with $M(t)\in P_\a(\mathcal{H})$ for every $t\in[0,+\infty)$. If the following two conditions are fulliled

(i)  the limit $\lim_{t\rightarrow+\infty}\|x(t)-z\|_{M(t)}$ exists for every $z\in \mathcal{C}$;

(ii) every weak sequential cluster point of $x(t), t \in [0,+\infty),$ belongs to $\mathcal{C}$;

\noindent then there exists $x_{\infty}\in \mathcal{C}$ such that $x(t), t \in [0,+\infty)$, converges weakly to $x_{\infty}$ as $t \rightarrow +\infty$.
\end{lemma}
\begin{proof} Since $\mathcal{C}\neq\emptyset$ and  $M(t)\in P_\a(\mathcal{H})$, by (i) we have that $x$ is bounded, hence it possesses at least one weak sequential cluster point, which belongs to ${\cal C}$. We show that $x$ has exactly one weak sequential cluster point.

 Indeed, let $x_1,x_2$ two weak sequential cluster points of $x$. For our claim it is enough to show that $x_1=x_2$. Obviously $x_1,x_2\in \mathcal{C}$ and there exist the sequences $(t_n^1)_{n\geq 0},(t_n^2)_{n\geq 0}\subseteq[0,+\infty)$ with $\lim_{n\To+\infty}t_n^1=+\infty$ and $\lim_{n\To+\infty}t_n^2=+\infty$ such that $(x(t_n^1))_{n \geq 0}$ converges weakly to $x_1$ and $(x(t_n^2))_{n \geq 0}$ converges weakly to $x_2$ as $n \rightarrow +\infty$.

Further, since  $M(t_1)\succcurlyeq M(t_2)$ for every  for every $t_1,t_2\in [0,+\infty)$ with $t_1\le t_2$ and $M(t)\in P_\a(\mathcal{H})$  for every $t \in [0,+\infty),$ it follows that  for every $z\in \mathcal{H}$ the function
$$[0,+\infty) \rightarrow [0,+\infty), t\To\|z\|_{M(t)}^2,$$
 is decreasing and is bounded from below, hence there exists
 \begin{equation}\label{limnorm}\lim_{t\To+\infty}\|z\|_{M(t)}^2\in\R.
 \end{equation}

Since $x_1,x_2\in \mathcal{C}$, we have that  the limits $\lim_{t\rightarrow+\infty}\|x(t)-x_1\|_{M(t)}^2$ and  $\lim_{t\rightarrow+\infty}\|x(t)-x_2\|_{M(t)}^2$ exist. Further, since
$$-\<x(t),M(t)(x_1-x_2)\>= \frac12\left(\|x(t)-x_1\|_{M(t)}^2-\|x(t)-x_2\|_{M(t)}^2-\|x_1\|_{M(t)}^2+\|x_2\|_{M(t)}^2\right)$$
holds for every $t \in [0,+\infty)$, the limit
\begin{equation}\label{limit} \lambda:=\lim_{t\To+\infty}\<x(t),M(t)(x_1-x_2)\> \in\R
\end{equation}
exists.

Next we show that the limits
\begin{equation}\label{zlimit}
\lim_{t\To+\infty}M(t)z
\end{equation}
exists for every $z \in {\cal H}$. To this end we fix $z \in {\cal H}$. We will actually show that
$$\lim\limits_{s,t \rightarrow +\infty}  \|M(t)z-M(s)z\| = 0$$
and the conclusion will follow by the Cauchy criterion.

For $U \in S_+(\mathcal{H})$ we have by the generalized Cauchy-Schwarz inequality that 
$$|\<Ux,z\>| \le\|x\|_U \|z\|_U \mbox{ for every }x,z \in H.$$
Hence, for $t,s \in [0,+\infty)$ with $t \leq s$ we have $M(t)-M(s)\in S_+(\mathcal{H})$, therefore 
\begin{align*}
 \|(M(t)-M(s))z\|^2 & =\<(M(t)-M(s))z,M(t)-M(s))z\> \\
& \le\|z\|_{(M(t)-M(s))} \|(M(t)-M(s))z\|_{(M(t)-M(s))} \\
& = \|z\|_{(M(t)-M(s))} \left( \<(M(t)-M(s))^2z,(M(t)-M(s))z\>\right)^\frac{1}{2}\\
& \le \|z\|_{(M(t)-M(s))} \|M(t)-M(s)\|^\frac{3}{2}\|z\|.
\end{align*}
Since $M(0)\succcurlyeq M(t)$, we have that
$$\|M(t)\|=\sup_{\|x\|=1}\<M(t)x,x\>\le \sup_{\|x\|=1}\<M(0)x,x\>\le \|M(0)\|$$
for every $t\in [0,+\infty)$. This shows that $\|M(t)-M(s)\|, t,s \in [0,+\infty),$ is bounded. This, together with the fact that $\lim\limits_{s,t \rightarrow +\infty} \|z\|_{(M(t)-M(s))}^2 = 0$, which follows from \eqref{limnorm}, implies
$$\|(M(t)-M(s))z\|\To 0 \ \mbox{as} \ s,t\To+\infty.$$
This proves \eqref{zlimit}. For every  $z\in {\cal H}$ let us denote by $Mz:=\lim_{t\To+\infty}M(t)z.$ Since $M(t)\in P_\a(\mathcal{H})$ for every $t \in [0,+\infty)$, it holds
$$M\in P_\a(\mathcal{H}).$$

Since $(x(t_n^1))_{n \geq 0}$ converges weakly to $x_1$ and $(x(t_n^2))_{n \geq 0}$ converges weakly to $x_2$ as $n \rightarrow +\infty$
and
$$M(t_n^1)(x_2-x_1)\To M(x_2-x_1)\mbox{ and }M(t_n^2)(x_2-x_1)\To M(x_2-x_1) \ \mbox{as} \ n \To+\infty,$$
passing to the limit in \eqref{limit} we get
$$\lim_{n\To+\infty}\<x(t_n^1),M(t_n^1)(x_2-x_1)\>=\<x_1,M(x_2-x_1\>)=\l$$
and
$$\lim_{n\To+\infty}\<x(t_n^2),M(t_n^2)(x_2-x_1)\>=\<x_2,M(x_2-x_1)\>=\l.$$
In conclusion,
$$0=\<x_2,M(x_2-x_1)\>-\<x_1,M(x_2-x_1)\>=\|x_2-x_1\|_M^2\ge\a\|x_2-x_1\|^2,$$
which shows that $x_1=x_2.$
\end{proof}

\begin{remark}\label{decresingM}
If a map $M:[0,+\infty)\To S_+(\mathcal{H})$ satisfies $M(t_1)\succcurlyeq M(t_2)$ for every $t_1,t_2\in [0,+\infty)$ with $t_1\le t_2$ we say that $M$ is monotonically decreasing. If  $M$ is monotonically decreasing and locally absolutely continuous, then $\dot{M}(t)$ exists and $\<\dot{M}(t)x,x\>\le 0$ for almost every $t\in[0,+\infty).$
\end{remark}

The following result is an adaptation of a result from \cite{ACS} to our setting.

\begin{proposition}(see \cite[Proposition 2.4]{ACS})\label{weakconv} In the setting of the optimization problem \eqref{primal}, let $(a_n, a_n^*)_{n\geq 0}$ be a sequence
in the graph of $\p (f+h)$ and $(b_n, b_n^*)_{n\geq 0}$ a sequence in the graph of $\p g.$  Suppose that
$a_n$ converges weakly to $\ol x\in \mathcal{H},$  $b_n^*$ converges weakly to $\ol v\in\mathcal{G}$,
$a_n^*+A^* b_n^*\To 0, $ and $A a_n-b_n \To 0$ as $n \rightarrow +\infty$. Then
$$\<a_n,a_n^*\>+\<b_n,b_n^*\>\To 0 \ \mbox{as} \ n \rightarrow +\infty$$ and
$$\ol v\in \p g(A\ol x),\, -A^*\ol v \in\p f(\ol x) + \n h(\ol x).$$
\end{proposition}

The theorem which states the asymptotic convergence of the trajectories generated by the dynamical system \eqref{ADMMdysy-subdiff} to a saddle point of the Lagrangian of the problem \eqref{primal} follows.
\begin{theorem}\label{convergence} In the setting of the optimization problem \eqref{primal}, assume that the set of  saddle points of the Lagrangian $l$ is nonempty,
the maps
$$[0,+\infty) \rightarrow S_+({\cal H}), t \mapsto M_1(t), \ \mbox{and} \ [0,+\infty) \rightarrow S_+({\cal G}), t \mapsto M_2(t),$$
are locally absolutely continuous and monotonically decreasing, 
$$ M_1(t)+\frac{c(1-\g)}{4}A^*A-\frac{L_h}{4}I \in S_+({\cal H}) \quad  \forall t \in [0,+\infty),$$
and
$$\sup_{t\ge 0}\|\dot{M_1}(t)\|<+\infty \ \mbox{and} \ \sup_{t\ge 0}\|\dot{M_2}(t)\|<+\infty.$$
For an arbitrary starting point $(x^0,z^0,y^0)\in {\mathcal H}\times{\mathcal G}\times{\mathcal G}$, let $(x,z,y):[0,+\infty)\To {\mathcal H}\times{\mathcal G}\times{\mathcal G}$ 
be the unique strong global solution of the dynamical system \eqref{ADMMdysy-subdiff}. If one of the following conditions holds:
\begin{itemize}
\item[(I)] there exists $\alpha > 0$ such that $M_1(t)+\frac{c(1-\g)}{4}A^*A-\frac{L_h}{4}I\in P_{\a}(\mathcal{H})$ for every $t \in [0,+\infty)$;
\item[(II)] $\g\in[0,1)$ and there exists $\alpha > 0$ such that $A^*A\in P_\a(\mathcal{H})$;
\end{itemize}
then the trajectory $(x(t),z(t),y(t))$ converges weakly to a saddle point of $l$ as $t\To+\infty.$
\end{theorem}

\begin{proof} The proof of the theorem relies on  Lemma \ref{opial-contvar}. An important step in the proof will be the derivation of two inequalities of Lyapunov type, namely, 
\eqref{derivineq}, in the case when $L_h \neq 0$, and \eqref{derivineq0}, in the case when $L_h=0$. Let $(x^*,z^*,y^*)\in{\mathcal H}\times{\mathcal G}\times{\mathcal G}$ be a saddle point of the Lagrangian $l$. Then
$$\left \{ \begin{array}{l}
0 \in \partial f(x^*) + \nabla h(x^*) + A^*y^*\\
Ax^* = z^*, Ax^* \in \partial g^*(y^*).
\end{array} \right. $$

According to \eqref{1eq} we have for almost every $t \in [0,+\infty)$
\begin{equation}\label{use3}
-c A^*A(\dot{x}(t)+x(t))-M_1(t)\dot{x}(t)+c A^*z(t)-A^*y(t)-\n h(x(t))\in\p f(\dot{x}(t)+x(t)),
\end{equation}
which yields, by taking into account the monotonicity of $\partial f$,
\begin{equation}\label{e1}
\<-c A^*A(\dot{x}(t)+x(t))-M_1(t)\dot{x}(t)+c A^*z(t)-A^*(y(t)-y^*)-(\n h(x(t))-\n h(x^*)),\dot{x}(t)+x(t)-x^*\>\ge 0.
\end{equation}
Similarly, according to \eqref{2eq} we have for almost every $t \in [0,+\infty)$
\begin{equation}\label{use4}
-c(\dot{z}(t)+z(t))+cA(\g\dot{x}(t)+x(t))-M_2(t)\dot{z}(t)+y(t)\in\p g(\dot{z}(t)+z(t)),
\end{equation}
which yields, by taking into account the monotonicity of $\partial g$,
\begin{equation}\label{e2}
\<-c(\dot{z}(t)+z(t))+cA(\g\dot{x}(t)+x(t))-M_2(t)\dot{z}(t)+(y(t)-y^*),\dot{z}(t)+z(t)-Ax^*\>\ge 0.
\end{equation}
By using the last equation of  \eqref{ADMMdysy-subdiff} we obtain for almost every $t \in [0,+\infty)$
\begin{align}\label{e3}
& \<- A^*(y(t)-y^*),\dot{x}(t)+x(t)-x^*\>+\<y(t)-y^*,\dot{z}(t)+z(t)-Ax^*\> \nonumber \\
= & -\<y(t)-y^*,A(\dot{x}(t)+x(t))-Ax^*-(\dot{z}(t)+z(t))+Ax^*\> = -\frac{1}{c}\<y(t)-y^*,\dot{y}(t)\>\\
= & -\frac{1}{2c}\frac{d}{dt}\|y(t)-y^*\|^2. \nonumber
\end{align}
Assume that $L_h >0$. By using the Baillon-Haddad Theorem we have for almost every $t \in [0,+\infty)$
\begin{align}\label{e31}
& \<-(\n h(x(t))-\n h(x^*)),\dot{x}(t)+x(t)-x^*\> \nonumber \\
= & -\<\n h(x(t))-\n h(x^*),x(t)-x^*\>-\<\n h(x(t))-\n h(x^*),\dot{x}(t)\> \nonumber \\
\le & -\frac{1}{L_h}\|\n h(x(t))-\n h(x^*)\|^2-\<\n h(x(t))-\n h(x^*),\dot{x}(t)\>\\
= & -\frac{1}{L_h}\left(\left\|\n h(x(t))-\n h(x^*)+\frac{L_h}{2}\dot{x}(t)\right\|^2-\frac{L_h^2}{4}\|\dot{x}(t)\|^2\right).\nonumber
\end{align}
By summing \eqref{e1} and \eqref{e2} and by taking into account \eqref{e3} and \eqref{e31} we obtain for almost every $t \in [0,+\infty)$
\begin{align}\label{e411}
0 \leq & \ \<-c A^*A(\dot{x}(t)+x(t))-M_1(t)\dot{x}(t)+c A^*z(t),\dot{x}(t)+x(t)-x^*\> \nonumber\\
& + \<-c(\dot{z}(t)+z(t))+cAx(t)+c\g A\dot{x}(t)-M_2(t)\dot{z}(t),\dot{z}(t)+z(t)-Ax^*\>\\
& -\frac{1}{2c}\frac{d}{dt}\|y(t)-y^*\|^2 -\frac{1}{L_h}\left(\left\|\n h(x(t))-\n h(x^*)+\frac{L_h}{2}\dot{x}(t)\right\|^2-\frac{L_h^2}{4}\|\dot{x}(t)\|^2\right).\nonumber
\end{align}
We have for almost every $t \in [0,+\infty)$
\begin{align*}
& \<-c A^*A(\dot{x}(t)+x(t))+c A^*z(t),\dot{x}(t)+x(t)-x^*\>\\
& +\<-c(\dot{z}(t)+z(t)) +cAx(t)+c\g A\dot{x}(t),\dot{z}(t)+z(t)-Ax^*\> \\
= & \ -\frac1c\|\dot{y}(t)\|^2+\<-c\dot{z}(t),A(\dot{x}(t)+x(t)-x^*)\>+\<(\g-1)cA\dot{x}(t),\dot{z}(t)+z(t)-Ax^*\>\\
= & -\frac1c\|\dot{y}(t)\|^2+\left\<-c\dot{z}(t),\frac1c\dot{y}(t)+\dot{z}(t)+z(t)-Ax^*\right\>\\
& +\left\<(\g-1)cA\dot{x}(t),A(\dot{x}(t)+x(t))-\frac1c\dot{y}(t)-Ax^*\right\>\\
= & -\frac1c\|\dot{y}(t)\|^2-c\|\dot{z}(t)\|^2+(\g-1)c\|A\dot{x}(t)\|^2-\<\dot{z}(t),\dot{y}(t)\>+(1-\g)\<A\dot{x}(t),\dot{y}(t)\>\\
& + \frac{c(\g-1)}{2}\frac{d}{dt}\left(\|Ax(t)-Ax^*\|^2\right)-\frac{c}{2}\frac{d}{dt}\left(\|z(t)-Ax^*\|^2\right).
\end{align*}
Since
$$\<\dot{z}(t),\dot{y}(t)\>=\left\|\frac{\sqrt{3c}}{2}\dot{z}(t)+\frac{1}{\sqrt{3c}}\dot{y}(t)\right\|^2- \frac{3c}{4}\|\dot{z}(t)\|^2- \frac{1}{3c}\|\dot{y}(t)\|^2,$$
and 
$$\<A\dot{x}(t),\dot{y}(t)\>=-\left\|\frac{\sqrt{3c}}{2}A\dot{x}(t)-\frac{1}{\sqrt{3c}}\dot{y}(t)\right\|^2+\frac{3c}{4}\|A\dot{x}(t)\|^2+\frac{1}{3c}\|\dot{y}(t)\|^2,$$
we obtain from above that for almost every $t \in [0,+\infty)$ it holds
\begin{align}\label{E1}
& \<-c A^*A(\dot{x}(t)+x(t))+c A^*z(t),\dot{x}(t)+x(t)-x^*\> \nonumber\\
& +\<-c(\dot{z}(t)+z(t)) +cAx(t)+c\g A\dot{x}(t),\dot{z}(t)+z(t)-Ax^*\> \nonumber\\
= & \ -\frac{\g+1}{3c}\|\dot{y}(t)\|^2-\frac{c}{4}\|\dot{z}(t)\|^2-\frac{(1-\g)c}{4}\|A\dot{x}(t)\|^2 -\left\|\frac{\sqrt{3c}}{2}\dot{z}(t)+\frac{1}{\sqrt{3c}}\dot{y}(t)\right\|^2\\
& -(1-\g)\left\|\frac{\sqrt{3c}}{2}A\dot{x}(t)-\frac{1}{\sqrt{3c}}\dot{y}(t)\right\|^2 +\frac{c(\g-1)}{2}\frac{d}{dt}\left(\|Ax(t)-Ax^*\|^2\right)-\frac{c}{2}\frac{d}{dt}\left(\|z(t)-Ax^*\|^2\right). \nonumber
\end{align}
By using Lemma \ref{compderiv} we observe that for almost every $t\in[0,+\infty)$ it holds
\begin{align*}
\<-M_1(t)\dot{x}(t),\dot{x}(t)+x(t)-x^*\> = & -\|\dot{x}(t)\|_{M_1(t)}^2-\<M_1(t)\dot{x}(t),x(t)-x^*\>\\
=&  -\|\dot{x}(t)\|_{M_1(t)}^2+\frac12 \<\dot M_1(t)(x(t)-x^*), x(t)-x^*\>\\
& -\frac12\frac{d}{dt}\|x(t)-x^*\|_{M_1(t)}^2
\end{align*}
and
\begin{align*}
\<-M_2(t)\dot{z}(t),\dot{z}(t)+z(t)-Ax^*\>= & -\|\dot{z}(t)\|_{M_2(t)}^2-\<M_2(t)\dot{z}(t),z(t)-Ax^*\>\\
= &-\|\dot{z}(t)\|_{M_2(t)}^2+\frac12 \<\dot M_2(t)(z(t)-Ax^*), z(t)-Ax^*\> \\
&-\frac12\frac{d}{dt}\|z(t)-Ax^*\|_{M_2(t)}^2.
\end{align*}

By plugging the last two identities and \eqref{E1} into \eqref{e411}, we obtain for almost every $t\in[0,+\infty)$
\begin{align*}
0 \leq & -\frac{\g+1}{3c}\|\dot{y}(t)\|^2-\frac{c}{4}\|\dot{z}(t)\|^2-\frac{(1-\g)c}{4}\|A\dot{x}(t)\|^2 -\left\|\frac{\sqrt{3c}}{2}\dot{z}(t)+\frac{1}{\sqrt{3c}}\dot{y}(t)\right\|^2\\
& -(1-\g)\left\|\frac{\sqrt{3c}}{2}A\dot{x}(t)-\frac{1}{\sqrt{3c}}\dot{y}(t)\right\|^2 + \frac{c(\g-1)}{2}\frac{d}{dt}\left(\|Ax(t)-Ax^*\|^2\right) -\frac{c}{2}\frac{d}{dt}\left(\|z(t)-Ax^*\|^2\right)\\
& -\|\dot{x}(t)\|_{M_1(t)}^2+\frac12 \<\dot M_1(t)(x(t)-x^*), x(t)-x^*\> -\frac12\frac{d}{dt}\|x(t)-x^*\|_{M_1(t)}^2\\
& -\|\dot{z}(t)\|_{M_2(t)}^2+\frac12 \<\dot M_2(t)(z(t)-Ax^*), z(t)-Ax^*\> -\frac12\frac{d}{dt}\|z(t)-Ax^*\|_{M_2(t)}^2\\
& -\frac{1}{2c}\frac{d}{dt}\|y(t)-y^*\|^2 -\frac{1}{L_h}\left(\left\|\n h(x(t))-\n h(x^*)+\frac{L_h}{2}\dot{x}(t)\right\|^2-\frac{L_h^2}{4}\|\dot{x}(t)\|^2\right).
\end{align*}
According to Remark \ref{decresingM}, 
$$\<\dot M_1(t)(x(t)-x^*), x(t)-x^*\> \leq 0 \ \mbox{and} \ \<\dot M_2(t)(z(t)-Ax^*), z(t)-Ax^*\> \leq 0$$
for almost every $t\in[0,+\infty)$. This means that for almost every $t\in[0,+\infty)$ we have
\begin{align}\label{derivineq}
\frac12\frac{d}{dt}\left(\|x(t)-x^*\|_{M_1(t)+c(1-\g)A^*A}^2+\|z(t)-Ax^*\|_{M_2(t)+cI}^2+\frac1c\|y(t)-y^*\|^2\right)+ & \nonumber \\
\|\dot{x}(t)\|_{M_1(t)+\frac{(1-\g)c}{4}A^*A-\frac{L_h}{4}I}^2+\|\dot{z}(t)\|_{M_2(t)+\frac{c}{4}I}^2+\frac{\g+1}{3c}\|\dot{y}(t)\|^2+ & \nonumber\\
\left\|\frac{\sqrt{3c}}{2}\dot{z}(t)+\frac{1}{\sqrt{3c}}\dot{y}(t)\right\|^2+(1-\g)\left\|\frac{\sqrt{3c}}{2}A\dot{x}(t)-\frac{1}{\sqrt{3c}}\dot{y}(t)\right\|^2+& \\
\frac{1}{L_h}\left\|\n h(x(t))-\n h(x^*)+\frac{L_h}{2}\dot{x}(t)\right\|^2 &\le 0.\nonumber
\end{align}
From Lemma \ref{fejer-cont1} we have
\begin{equation}\label{finitelimit}
\lim_{t\To+\infty}(\|x(t)-x^*\|_{M_1(t)+c(1-\g)A^*A}^2+\|z(t)-Ax^*\|_{M_2(t)+cI}^2+\frac1c\|y(t)-y^*\|^2)\in \R.
\end{equation}
Let be $T >0$. By integrating \eqref{derivineq} on the interval $[0,T]$ we obtain
\allowdisplaybreaks
\begin{align*}
\frac{1}{2}\left(\|x(T)-x^*\|_{M_1(T)+c(1-\g)A^*A}^2+\|z(T)-z^*\|_{M_2(T)+cI}^2+\frac{1}{c}\|y(T)-y^*\|^2\right)+&\\
\int_0^T\|\dot{x}(t)\|_{M_1(t)+\frac{(1-\g)c}{4}A^*A-\frac{L_h}{4}I}^2dt+\int_0^T\|\dot{z}(t)\|_{M_2(t)+\frac{c}{4}I}^2dt+\frac{\g+1}{3c}\int_0^T\|\dot{y}(t)\|^2dt+&\\
\int_0^T\left\|\frac{\sqrt{3c}}{2}\dot{z}(t)+\frac{1}{\sqrt{3c}}\dot{y}(t)\right\|^2dt+(1-\g)\int_0^T\left\|\frac{\sqrt{3c}}{2}A\dot{x}(t)-\frac{1}{\sqrt{3c}}\dot{y}(t)\right\|^2dt+&\\
\frac{1}{L_h}\int_0^T\left\|\n h(x(t))-\n h(x^*)+\frac{L_h}{2}\dot{x}(t)\right\|^2dt& \leq\\
\frac{1}{2}\left(\|x_0-x^*\|_{M_1(0)+c(1-\g)A^*A}^2+\|z_0-z^*\|_{M_2(0)+cI}^2+\frac{1}{c}\|y_0-y^*\|^2\right).
\end{align*}
Letting $T$ converge to $+\infty$ we find
\begin{align}
& \|\dot{x}(\cdot)\|_{M_1(\cdot)+\frac{(1-\g)c}{4}A^*A-\frac{L_h}{4}I}^2\in L^1([0,+\infty),\R), \label{int1}\\
& \|\dot{z}(\cdot)\|_{M_2(\cdot)+\frac{c}{4}I}^2\in L^1([0,+\infty),\R), \dot{y}(\cdot)\in L^2([0,+\infty),{\mathcal G}),\label{int2}\\
& \frac{\sqrt{3c}}{2}\dot{z}(\cdot)+\frac{1}{\sqrt{3c}}\dot{y}(\cdot) \in L^2([0,+\infty),{\mathcal G}), (1-\g)\left(\frac{\sqrt{3c}}{2}A\dot{x}(\cdot)-\frac{1}{\sqrt{3c}}\dot{y}(\cdot)\right)\in L^2([0,+\infty),{\mathcal G}),\label{int3}
\end{align}
and, consequently,
\begin{equation}\label{int4}
\dot{z}(\cdot),\,(1-\g)A\dot{x}(\cdot)\in L^2([0,+\infty),{\mathcal G}).
\end{equation}
In the case when $L_h=0$, which corresponds to the situation when $h$ is an affine-continuous function, instead of \eqref{derivineq} we obtain that for almost
every $t\in[0,+\infty)$
\begin{align}\label{derivineq0}
\frac12\frac{d}{dt}\left(\|x(t)-x^*\|_{M_1(t)+c(1-\g)A^*A}^2+\|z(t)-Ax^*\|_{M_2(t)+cI}^2+\frac1c\|y(t)-y^*\|^2\right)+ & \nonumber \\
\|\dot{x}(t)\|_{M_1(t)+\frac{(1-\g)c}{4}A^*A}^2+\|\dot{z}(t)\|_{M_2(t)+\frac{c}{4}I}^2+\frac{\g+1}{3c}\|\dot{y}(t)\|^2+ & \\
\left\|\frac{\sqrt{3c}}{2}\dot{z}(t)+\frac{1}{\sqrt{3c}}\dot{y}(t)\right\|^2+(1-\g)\left\|\frac{\sqrt{3c}}{2}A\dot{x}(t)-\frac{1}{\sqrt{3c}}\dot{y}(t)\right\|^2& \leq 0. \nonumber
\end{align}
By arguing as above, we obtain also in this case \eqref{finitelimit}, \eqref{int1}-\eqref{int3} and \eqref{int4}.

Further, we have that $\dot{x}(\cdot)\in L^2([0,+\infty),{\mathcal H})$. Indeed,  in case (I), when we assume that there exists $\alpha >0$ such that $M_1(t)+\frac{(1-\g)c}{4}A^*A-\frac{L_h}{4}I\in P_{\a}(\mathcal{H})$ for every $t\in[0,+\infty)$, then this 
yields automatically. In case (II), from $(1-\g)A\dot{x}(\cdot)\in L^2([0,+\infty),{\mathcal G})$ and $\g\in[0,1)$, we have
$$A\dot{x}(\cdot)\in L^2([0,+\infty),{\mathcal G}).$$
But, since $A^*A\in P_{\a}(H)$, it yields $\|A\dot{x}(t)\|^2\ge\a\|\dot{x}(t)\|^2$  for almost
every $t\in[0,+\infty)$, which means that also in this case
$$\dot{x}(\cdot)\in  L^2([0,+\infty),{\mathcal H}).$$
According to Lemma \ref{abscont}, this yields
$$\ddot{x}(\cdot)\in L^2([0,+\infty),{\mathcal H})\mbox{ and } \ddot{z}(\cdot),\, \ddot{y}(\cdot)\in L^2([0,+\infty),{\mathcal G}).$$

Consequently, for almost every $t\in[0,+\infty)$ it holds
$$\frac{d}{dt}\|\dot{x}(t)\|^2=2\<\ddot{x}(t),\dot{x}(t)\>\le\left(\|\ddot{x}(t)\|^2+\|\dot{x}(t)\|^2\right)$$
and the right-hand side is a function in $L^1([0,+\infty),\R)$. Hence, according to Lemma \ref{fejer-cont2},
$$\lim_{t\To+\infty}\dot{x}(t)=0.$$
Similarly, we obtain that
$$\lim_{t\To+\infty}\dot{z}(t)=0\mbox{ and }\lim_{t\To+\infty}\dot{y}(t)=0.$$

We will close the proof of the theorem by showing that the asymptotic convergence of the trajectory follows from Lemma \ref{opial-contvar}. One can easily notice that \eqref{finitelimit} is nothing else but condition (i) of this lemma when applied in the product space  for the trajectory
$$[0,+\infty) \mapsto  \mathcal{H}\times\mathcal{G}\times\mathcal{G},  \quad t\To(x(t),z(t),y(t)),$$
 the monotonically decreasing map
$$W: [0,+\infty) \mapsto  \mathcal{H}\times\mathcal{G}\times\mathcal{G},  \quad  W(t)=\left(M_1(t)+c(1-\g)A^*A,M_2(t)+cI,\frac{1}{c}I\right)$$
and the set $\mathcal{C}$ taken as the set of saddle points of the Lagrangian $l.$

Next we will show that also condition (ii) in  Lemma \ref{opial-contvar} is fulfilled, namely, that every weak sequential cluster point of the trajectory $(x(t),z(t),y(t)), t \in [0,+\infty),$ is a saddle point of the Langrangian $l$.

Let $(\ol x,\ol z,\ol y)$ be such a weak sequentially cluster point. This means that there exists a sequence $(s_n)_{n \geq 0}$ with  $s_n\To+\infty$ such that  $(x(s_n),z(s_n),y(s_n))$ converges to $(\ol x,\ol z,\ol y)$  as $n\To+\infty$ in the weak topology of $\mathcal{H}\times\mathcal{G}\times\mathcal{G}$.

From \eqref{use3} and \eqref{use4}  we get for every $n \geq 0$
$$-c A^*A(\dot{x}(s_n)+x(s_n))-M_1(s_n)\dot{x}(s_n)+c A^*z(s_n)-A^*y(s_n)-\n h(x(s_n))\in\p f(\dot{x}(s_n)+x(s_n))$$
and
$$-c(\dot{z}(s_n)+z(s_n))+cA(\g\dot{x}(s_n)+x(s_n))-M_2(s_n)\dot{z}(s_n)+y(s_n)\in\p g(\dot{z}(s_n)+z(s_n)),$$
respectively. For every $n \geq 0$, let
$$a_n^*:=-c A^*A(\dot{x}(s_n)+x(s_n))-M_1(s_n)\dot{x}(s_n)+c A^*z(s_n)-A^*y(s_n)-\n h(x(s_n))+\n h(\dot{x}(s_n)+x(s_n))$$
and
$$a_n:=\dot{x}(s_n)+x(s_n).$$
Hence, $(a_n,a_n^*)_{n \geq 0}\subseteq \gr\p (f+h).$ Similarly, for every $n \geq 0$, let
$$b_n^*:=-c(\dot{z}(s_n)+z(s_n))+cA(\g\dot{x}(s_n)+x(s_n))-M_2(s_n)\dot{z}(s_n)+y(s_n)$$
and
$$b_n:=\dot{z}(s_n)+z(s_n).$$
Hence, $(b_n,b_n^*)_{n \geq 0}\subseteq\gr\p g.$

Since $\lim_{t\To+\infty}\dot{x}(t)=0$, $\lim_{t\To+\infty}\dot{z}(t)=0$ and  $\lim_{t\To+\infty}\dot{y}(t)=0$ it follows that $(a_n)_{n \geq 0}$  converges weakly to $\ol x$ as $n \rightarrow \infty$. Furthermore, since $(M_2(s_n))_{n \geq 0}$ is bounded, and 
$$b_n^*=c(\g-1)A\dot{x}(s_n)+\dot{y}(s_n)-M_2(s_n)\dot{z}(s_n)+y(s_n) \ \forall n \geq 0,$$ it follows that $(b_n^*)_{n \geq 0}$  converges weakly to $\ol y$ as $n \rightarrow \infty$.

From \eqref{ADMMdysy} we have
$$A a_n-b_n=\frac{1}{c}\dot{y}(s_n)\To 0 \ (n \rightarrow +\infty),$$
which implies that $A\ol x=\ol z$. On the other hand, since $\n h$ is Lipschitz continuous, we have
$$\n h(\dot{x}(s_n)+x(s_n))- \n h(x(s_n))\To 0 \ (n \rightarrow +\infty),$$
hence
\begin{align*}
& \ \lim_{n\To+\infty}(a_n^*+A^*b_n^*)\\
= & \ \lim_{n\To+\infty}(c(\g-1)A^*A\dot{x}(s_n)-cA^*\dot{z}(s_n)-A^*M_2(s_n)\dot{z}(s_n)-M_1(s_n)\dot{x}(s_n))\\
& + \lim_{n\To+\infty}(\n h(\dot{x}(s_n)+x(s_n))- \n h(x(s_n)))\\
= & \ 0.
\end{align*}
Thus, according to Proposition \ref{weakconv}, we have
$$- A^*\ol y-\n h(\ol x)\in\p f(\ol x) \ \mbox{and} \
 \ol y\in\p g(A\ol x).$$
Consequently,
$(\ol x, \ol z, \ol y)$ is a saddle point of $l.$ 

The conclusion of the theorem follows from Lemma \ref{opial-contvar}. 
\end{proof}

Next we will address two particular cases of the dynamical system \eqref{ADMMdysy-subdiff}. We consider first the case when $M_1(t)=M_2(t)=0$ for every $t\in [0, +\infty)$, thus, the system \eqref{ADMMdysy-subdiff} becomes 
\begin{equation}\label{ADMMdysyM}
\left\{
\begin{array}{llll}
\dot{x}(t)+x(t)\in\argmin\limits_{x\in\mathcal{H}}\left(f(x)+ \langle x, \nabla h(x(t)) \rangle + \frac{c}{2}
\left \|Ax -z(t) + \frac{1}{c}y(t) \right\|^2 \right)\\
\\
\dot{z}(t)+z(t)=\argmin\limits_{x\in\mathcal{G}}\left(g(x)+\frac{c}{2}\left\|x-\left(A(\g\dot{x}(t)+x(t))+\frac{1}{c}y(t)\right)\right\|^2\right)\\
\\
\dot{y}(t)=c A(x(t)+\dot{x}(t))-c (z(t)+\dot{z}(t))\\
\\
x(0)=x^0\in{\mathcal H},\,z(0)=z^0\in{\mathcal G},\,y(0)=y^0 \in{\mathcal G},
\end{array}
\right.
\end{equation}
where $c >0$ and $\gamma \in [0,1]$. The dynamical system \eqref{ADMMdysyM} can be seen as the continuous counterpart of the classical ADMM algorithm. The corresponding convergence result follows as a particular case of  Theorem \ref{convergence}.

\begin{theorem}\label{convergenceM} In the setting of the optimization problem \eqref{primal}, assume that the set of  saddle points of the Lagrangian $l$ is nonempty, $\gamma \in [0,1)$ and that there exists $\alpha >0$ such that $A^*A-\frac{L_h}{c(1-\g)}I\in P_\a(\mathcal{H})$. For an arbitrary starting point $(x^0,z^0,y^0)\in {\mathcal H}\times{\mathcal G}\times{\mathcal G}$, let $(x,z,y):[0,+\infty)\To {\mathcal H}\times{\mathcal G}\times{\mathcal G}$ 
be the unique strong global solution of the dynamical system \eqref{ADMMdysyM}. Then the trajectory $(x(t),z(t),y(t))$ converges weakly to a saddle point of $l$ as $t\To+\infty.$
\end{theorem}

Next we consider the setting from Remark \ref{rem-cont-primal-dual} with $M_1(t) = \frac{1}{\tau(t)} I-cA^*A$ and $M_2(t) = 0$, where $\tau(t)$ is such that $c \tau(t) \|A\|^2 \leq 1$, for every $t\in [0, +\infty)$. The resulting dynamical system is the primal-dual system \eqref{ADMMdysyM1}. The corresponding convergence result follows again as a particular case of  Theorem \ref{convergence}.

\begin{theorem}\label{convergenceM1} In the setting of the optimization problem \eqref{primal}, assume that the set of  saddle points of the Lagrangian $l$ is nonempty,
the map $\tau : [0,+\infty) \rightarrow (0,+\infty)$ is locally absolutely continuous and monotonically increasing with 
$$c\tau(t)\|A\|^2 \leq 1 \ \mbox{and} \ \frac{4-\tau(t)L_h}{4\tau(t)}I-\frac{c(3+\gamma)}{4}A^*A\in S_+(\mathcal{H})  \quad  \forall t \in [0,+\infty),$$ 
and $\sup_{t\ge 0}\frac{\tau'(t)}{\tau^2(t)}<+\infty$.
For an arbitrary starting point $(x^0,z^0,y^0)\in {\mathcal H}\times{\mathcal G}\times{\mathcal G}$, let $(x,z,y):[0,+\infty)\To {\mathcal H}\times{\mathcal G}\times{\mathcal G}$ 
be the unique strong global solution of the dynamical system \eqref{ADMMdysyM1}. If one of the following assumptions holds:
\begin{itemize}
\item[(I)] there exists $\alpha > 0$ such that $\frac{4-\tau(t)L_h}{4\tau(t)}I-\frac{c(3+\gamma)}{4}A^*A\in P_{\a}(\mathcal{H})$ for every $t \in [0,+\infty)$;
\item[(II)] $\g\in[0,1)$ and there exists $\alpha > 0$ such that $A^*A\in P_\a(\mathcal{H})$;
\end{itemize}
then the trajectory $(x(t),z(t),y(t))$ converges weakly to a saddle point of $l$ as $t\To+\infty.$
\end{theorem}

\begin{remark}\label{rem-param}
Let be $t \in [0,+\infty)$. Notice that the condition $\frac{4-\tau(t)L_h}{4\tau(t)}I-\frac{c(3+\gamma)}{4}A^*A\in S_+(\mathcal{H})$ is fulfilled if and only if
$$\tau(t)\left(\frac{L_h}{4}+\frac{c(3+\gamma)}{4}\|A\|^2\right)\leq 1.$$ On the other hand, the condition 
$\frac{4-\tau(t)L_h}{4\tau(t)}I-\frac{c(3+\gamma)}{4}A^*A\in P_{\a}(\mathcal{H})$ holds, for $\alpha >0$, if and only if 
$$\tau(t)\left(\alpha +\frac{L_h}{4}+\frac{c(3+\gamma)}{4}\|A\|^2\right)\leq 1.$$  
\end{remark}

For the last result of this paper we go back to the general dynamical system \eqref{ADMMdysy-subdiff} and provide convergence rates for the violation of the feasibility condition by ergodic trajectories and the convergence of the objective function along these ergodic trajectories to its minimal value. 
The result can be seen as the continuous counterpart of a convergence rate result proved for the ADMM algorithm in \cite[Theorem 4.3]{cui-li-sun-toh}. 

\begin{theorem}\label{therg} In the setting of the optimization problem \eqref{primal}, assume that the set of saddle points of the Lagrangian $l$ is nonempty,
the maps
$$[0,+\infty) \rightarrow S_+({\cal H}), t \mapsto M_1(t), \ \mbox{and} \ [0,+\infty) \rightarrow S_+({\cal G}), t \mapsto M_2(t),$$
are locally absolutely continuous and monotonically decreasing, 
$$ M_1(t)+\frac{c(1-\g)}{4}A^*A-\frac{L_h}{2}I \in S_+({\cal H}) \quad  \forall t \in [0,+\infty),$$
$$\sup_{t\ge 0}\|\dot{M_1}(t)\|<+\infty \ \mbox{and} \ \sup_{t\ge 0}\|\dot{M_2}(t)\|<+\infty$$
and that one of the following conditions holds:
\begin{itemize}
\item[(I)] there exists $\alpha > 0$ such that $M_1(t)+\frac{c(1-\g)}{4}A^*A-\frac{L_h}{4}I\in P_{\a}(\mathcal{H})$ for every $t \in [0,+\infty)$;
\item[(II)] $\g\in[0,1)$ and there exists $\alpha > 0$ such that $A^*A\in P_\a(\mathcal{H})$;
\end{itemize}
For an arbitrary starting point $(x^0,z^0,y^0)\in {\mathcal H}\times{\mathcal G}\times{\mathcal G}$, let $(x,z,y):[0,+\infty)\To {\mathcal H}\times{\mathcal G}\times{\mathcal G}$ 
be the unique strong global solution of the dynamical system \eqref{ADMMdysy-subdiff}. Consider further for every $t \in (0,+\infty)$ the ergodic trajectories
$$\tilde{x}(t)=\frac{1}{t}\int_0^t (\dot{x}(s)+x(s))ds$$
and
$$\tilde{z}(t)=\frac{1}{t}\int_0^t(\dot{z}(s)+z(s))ds.$$
Then there exists $K\ge0$ such that for every $t \in (0,+\infty)$
$$\|A\tilde{x}(t)-\tilde{z}(t)\|\le\frac{K}{t}.$$
In addition, for every $\ol x\in\mathcal{H}$ and every $t \in (0, +\infty)$ such that $(\tilde{x}(t),\tilde{z}(t))\in\dom f\times \dom g$, one has
$$\Big((f+h)(\tilde{x}(t))+g(\tilde{z}(t))\Big)-\Big((f+h)(\ol x)+g(A \ol x)\Big)\le \frac{\|(x^0,z^0,y^0)-(\ol x,A \ol x,0)\|^2_{W(0)}}{2t},$$
where $$W(t)=\left(M_1(t)+c(1-\g)A^*A,M_2(t)+cI,\frac{1}{c}I\right).$$
\end{theorem}

\begin{proof}
Let $\ol x\in \mathcal{H}$ be fixed. By using \eqref{1eq}, that is
$$-c A^*A(\dot{x}(t)+x(t))-M_1(t)\dot{x}(t)+c A^*z(t)-A^*y(t)-\n h(x(t))\in\p f(\dot{x}(t)+x(t)),$$
it yields
\begin{equation}\label{forf}
f(\dot{x}(t)+x(t))-f(\ol x)\le\< c A^*A(\dot{x}(t)+x(t))+M_1(t)\dot{x}(t)-c A^*z(t)+A^*y(t)+\n h(x(t)),\ol x-(\dot{x}(t)+x(t))\>
\end{equation}
for almost every $t \in [0,+\infty)$. Similarly, by using \eqref{use4}, that is
$$-c(\dot{z}(t)+z(t))+cA(\g\dot{x}(t)+x(t))-M_2(t)\dot{z}(t)+y(t)\in\p g(\dot{z}(t)+z(t)),$$
it yields
\begin{equation}\label{forg}
g(\dot{z}(t)+z(t))-g(A\ol x)\le \<c(\dot{z}(t)+z(t))-cA(\g\dot{x}(t)+x(t))+M_2(t)\dot{z}(t)-y(t),A\ol x-(\dot{z}(t)+z(t))\>.
\end{equation}
for almost every $t \in [0,+\infty)$. Further, by using the convexity of $h$ and the Descent Lemma we obtain for almost every $t \in [0,+\infty)$
\begin{align}\label{forh}
h(\ol x)-h(\dot{x}(t)+x(t))-\<\n h(x(t)),\ol x-(\dot{x}(t)+x(t)) & \>\ge \nonumber\\
h(x(t))+\<\n h(x(t)),\ol x-x(t)\>-h(\dot{x}(t)+x(t))-\<\n h(x(t)),\ol x-(\dot{x}(t)+x(t))\> & = \\
h(x(t))-h(\dot{x}(t)+x(t))+\<\n h(x(t)),\dot{x}(t)\> & \ge-\frac{L_h}{2}\|\dot{x}(t)\|^2.  \nonumber
\end{align}
Adding \eqref{forf} and \eqref{forh} we obtain for almost every $t \in [0,+\infty)$
\begin{equation}\label{forfh}
\begin{array}{rl}
(f+h)(\dot{x}(t)+x(t))-(f+h)(\ol x) & \le \\
\< c A^*A(\dot{x}(t)+x(t))+M_1(t)\dot{x}(t)-c A^*z(t)+A^*y(t),\ol x-(\dot{x}(t)+x(t))\>+\frac{L_h}{2}\|\dot{x}(t)\|^2. &
\end{array}
\end{equation}
We recall the following four identities from the proof of Theorem \ref{convergence} (here were actually replace $x^*$ with $\ol x$ and $y^*$ by $0$)
\begin{align*}
& \<- A^*y(t),\dot{x}(t)+x(t)-\ol x\>+\<y(t),\dot{z}(t)+z(t)-A\ol x\>\\
= & -\<y(t),A(\dot{x}(t)+x(t))-A\ol x-(\dot{z}(t)+z(t))+A \ol x\> = \frac{1}{c}\<y(t),\dot{y}(t)\>\\
= & -\frac{1}{2c}\frac{d}{dt}\|y(t)\|^2,
\end{align*}
which corresponds to \eqref{e3},
\begin{align*}
& \<-c A^*A(\dot{x}(t)+x(t))+c A^*z(t),\dot{x}(t)+x(t)-\ol x\> \nonumber\\
& +\<-c(\dot{z}(t)+z(t)) +cAx(t)+c\g A\dot{x}(t),\dot{z}(t)+z(t)-A\ol x\> \nonumber\\
= & \ -\frac{\g+1}{3c}\|\dot{y}(t)\|^2-\frac{c}{4}\|\dot{z}(t)\|^2-\frac{(1-\g)c}{4}\|A\dot{x}(t)\|^2 -\left\|\frac{\sqrt{3c}}{2}\dot{z}(t)+\frac{1}{\sqrt{3c}}\dot{y}(t)\right\|^2\\
& -(1-\g)\left\|\frac{\sqrt{3c}}{2}A\dot{x}(t)-\frac{1}{\sqrt{3c}}\dot{y}(t)\right\|^2 +\frac{c(\g-1)}{2}\frac{d}{dt}\left(\|Ax(t)-A\ol x\|^2\right)-\frac{c}{2}\frac{d}{dt}\left(\|z(t)-A\ol x\|^2\right), \nonumber
\end{align*}
which corresponds to \eqref{E1}, and
\begin{align*}
\<-M_1(t)\dot{x}(t),\dot{x}(t)+x(t)-\ol x\> = & -\|\dot{x}(t)\|_{M_1(t)}^2-\<M_1(t)\dot{x}(t),x(t)-\ol x\>\\
=&  -\|\dot{x}(t)\|_{M_1(t)}^2+\frac12 \<\dot M_1(t)(x(t)-\ol x), x(t)-\ol x\>\\
& -\frac12\frac{d}{dt}\|x(t)-\ol x\|_{M_1(t)}^2
\end{align*}
and
\begin{align*}
\<-M_2(t)\dot{z}(t),\dot{z}(t)+z(t)-A\ol x\>= & -\|\dot{z}(t)\|_{M_2(t)}^2-\<M_2(t)\dot{z}(t),z(t)-A\ol x\>\\
= &-\|\dot{z}(t)\|_{M_2(t)}^2+\frac12 \<\dot M_2(t)(z(t)-A\ol x), z(t)-A\ol x\> \\
&-\frac12\frac{d}{dt}\|z(t)-A\ol x\|_{M_2(t)}^2,
\end{align*}
which all hold for for almost every $t \in [0,+\infty)$. By adding the four identities,  \eqref{forfh} and \eqref{forg}, we obtain for almost every $t \in [0,+\infty)$
\begin{align*}
\Big((f+h)(\dot{x}(t)+x(t))+g(\dot{z}(t)+z(t))\Big)-\Big((f+h)(\ol x)+g(A\ol x)\Big) & \le\\
-\frac{\g+1}{3c}\|\dot{y}(t)\|^2-\frac{c}{4}\|\dot{z}(t)\|^2-\frac{(1-\g)c}{4}\|A\dot{x}(t)\|^2&\\
-\left\|\frac{\sqrt{3c}}{2}\dot{z}(t)+\frac{1}{\sqrt{3c}}\dot{y}(t)\right\|^2 -(1-\g)\left\|\frac{\sqrt{3c}}{2}A\dot{x}(t)-\frac{1}{\sqrt{3c}}\dot{y}(t)\right\|^2&\\
-\frac{c(1-\g)}{2}\frac{d}{dt}\left(\|Ax(t)-A\ol x\|^2\right)-\frac{c}{2}\frac{d}{dt}\left(\|z(t)-A\ol x\|^2\right) + \frac{L_h}{2}\|\dot{x}(t)\|^2-\frac{1}{2c}\frac{d}{dt}\|y(t)\|^2&\\
-\|\dot{x}(t)\|_{M_1(t)}^2+\frac12 \<\dot M_1(t)(x(t)-\ol x), x(t)-\ol x\>-\frac12\frac{d}{dt}\|x(t)-\ol x\|_{M_1(t)}^2&\\
-\|\dot{z}(t)\|_{M_2(t)}^2+\frac12 \<\dot M_2(t)(z(t)-A\ol x), z(t)-A\ol x\>-\frac12\frac{d}{dt}\|z(t)-A\ol x\|_{M_2(t)}^2.&
\end{align*}
By neglecting the negative terms (here we use also that $M_1(t)+\frac{c(1-\g)}{4}A^*A-\frac{L_h}{2}I\in S_+(\mathcal{H})$), we obtain for almost every $t\in[0,+\infty)$ 
\begin{align}\label{forfghuse}
\Big((f+h)(\dot{x}(t)+x(t))+g(\dot{z}(t)+z(t))\Big)-\Big((f+h)(\ol x)+g(A\ol x)\Big) & \le \nonumber\\
-\frac12\frac{d}{dt}\left(\|x(t)-\ol x\|^2_{M_1(t)+c(1-\g)A^*A}+\|z(t)-A\ol x\|^2_{M_2(t)+cI}+\frac{1}{c}\|y(t)\|^2\right) & = \\
-\frac12\frac{d}{dt}\|(x(t),z(t),y(t))-(\ol x,A\ol x,0)\|^2_{W(t)}, \nonumber
\end{align}
where
$$W(t)=\left(M_1(t)+c(1-\g)A^*A,M_2(t)+cI,\frac{1}{c}I\right).$$
For $\tilde{x}(t)=\frac{1}{t}\int_0^t (\dot{x}(s)+x(s))ds$
and
$\tilde{z}(t)=\frac{1}{t}\int_0^t(\dot{z}(s)+z(s))ds,$
it holds
$$A\tilde{x}(t)-\tilde{z}(t)=\frac{1}{t}\int_0^t A(\dot{x}(s)+x(s))-(\dot{z}(s)+z(s))ds=\frac{1}{ct}\int_0^t \dot{y}(s) ds=\frac{y(t)-y_0}{ct} \ \forall t \in (0,+\infty).$$
From Theorem \ref{convergence} it follows that the trajectory $(x(t),z(t),y(t)) \to (x_{\infty},z_{\infty},y_{\infty}), t \in [0,+\infty)$, converges weakly to a saddle point of $l$ as $t \rightarrow +\infty$. This means that $y(t), t \in [0,+\infty),$ it is bounded, thus
there exists $K\ge 0$ such that
$$\|A\tilde{x}(t)-\tilde{z}(t)\|\le \frac{K}{t} \quad \forall t \in (0,+\infty).$$
Let $t \in (0, +\infty)$ be such that $(\tilde{x}(t),\tilde{z}(t))\in\dom f\times \dom g$. By Jensen's inequality in the integral form we have for every $t \in (0, +\infty)$
$$(f+h)(\tilde{x}(t))=(f+h)\left(\frac{1}{t}\int_0^t (\dot{x}(s)+x(s))ds\right)\le \frac{1}{t}\int_0^t (f+h)(\dot{x}(s)+x(s))ds$$
and
$$g(\tilde{z}(t))=g\left(\frac{1}{t}\int_0^t (\dot{z}(s)+z(s))ds\right)\le \frac{1}{t}\int_0^t g(\dot{z}(s)+z(s))ds,$$
which, combined with \eqref{forfghuse}, yields
\begin{align*}
(f+h)(\tilde{x}(t))+g(\tilde{z}(t)) & \le \\
\frac{1}{t}\int_0^t \Big ((f+h)(\dot{x}(s)+x(s))+g(\dot{z}(s)+z(s)) \Big)ds & \le \\
\frac{1}{t}\int_0^t \left(\Big((f+h)(\ol x)+g(A\ol x)\Big)-\frac12\frac{d}{ds}\|(x(s),z(s),y(s))-(\ol x,A\ol x,0)\|^2_{W(s)}\right)ds & = \\
(f+h)(\ol x)+g(A\ol x)- \!\frac{1}{2t}\!\left(\!\|(x(t),z(t),y(t))-\!(\ol x,A\ol x,0)\|^2_{W(t)}\!-\|(x(0),z(0),y(0))-\!(\ol x,A\ol x,0)\|^2_{W(0)}\!\right) & \le \\
(f+h)(\ol x)+g(A\ol x)+\frac{\|(x^0,z^0,y^0)-(\ol x,A\ol x,0)\|^2_{W(0)}}{2t}.&
\end{align*}
Hence,
\begin{equation*}
\Big((f+h)(\tilde{x}(t))+g(\tilde{z}(t))\Big)-\Big ((f+h)(\ol x)+g(A\ol x) \Big)\le \frac{\|(x^0,z^0,y^0)-(\ol x,A\ol x,0)\|^2_{W(0)}}{2t}.
\end{equation*}

\end{proof}

\end{document}